\documentclass[11pt, reqno]{amsart}
\usepackage{latexsym,amsmath,amssymb, amsthm, mathtools}
\usepackage{cases, verbatim, esint, todonotes}
\usepackage[active]{srcltx}
\usepackage[colorlinks]{hyperref}
\usepackage{hypernat}
\usepackage{color}

\newtheorem{thm}{Theorem}
\newtheorem{lem}[thm]{Lemma}
\newtheorem{prop}[thm]{Proposition}
\newtheorem{cor}[thm]{Corollary}

\theoremstyle{remark}
\newtheorem{rmk}[thm]{Remark}
\newtheorem{example}[thm]{Example}

\theoremstyle{definition}
\newtheorem{defi}[thm]{Definition}

\numberwithin{thm}{section} 
\numberwithin{equation}{section}

\makeatletter

\newcommand{\Rmnum}[1]{\expandafter\@slowromancap\romannumeral #1@}
\makeatother

\def\L{{\mathcal L}}

\def\R{{\mathbb R}}

\def\M{{\mathcal M}}

\def\S{{\mathbb S}}

\def\O{{\mathcal O}}

\newcommand{\pO}{\partial\Omega}
\newcommand{\Oba}{\overline{\Omega}}

\newcommand{\vep}{\varepsilon}
\newcommand{\ol}{\overline}

\newcommand{\dive}{\operatorname{div}}
\newcommand{\tr}{\operatorname{tr}}

\newcommand{\sgn}{\operatorname{sgn}}
\newcommand{\bpm}{\begin{pmatrix}}
\newcommand{\epm}{\end{pmatrix}}
\newcommand{\la}{\langle}
\newcommand{\ra}{\rangle}

\def\dist{{\rm dist\,}}

plus 6pt minus 12pt
\title[Parabolic Minkowski Convolution]{\protect{Parabolic Minkowski convolutions of viscosity solutions to fully nonlinear equations}}
\author[K. Ishige]{Kazuhiro Ishige}
\author[Q. Liu]{Qing Liu}
\author[P. Salani]{Paolo Salani}

\address[K. Ishige]{Graduate School of Mathematical Sciences, 
University of Tokyo 
3-8-1 Komaba, Meguro-ku, Tokyo, 153-8914, Japan. \textit{E-mail address: }{\tt ishige@ms.u-tokyo.ac.jp}}

\address[Q. Liu]{Department of Applied Mathematics, Faculty of Science, Fukuoka University, Fukuoka 814-0180, Japan. \textit{E-mail address: }{\tt qingliu@fukuoka-u.ac.jp}}

\address[P. Salani]{Dipartimento di Matematica``U. Dini”, Universit\`{a} di Firenze, viale Morgagni 67/A, 50134 Firenze, Italy. \textit{E-mail address: }{\tt paolo.salani@unifi.it}}

\date{\today}

\begin{document}

\begin{abstract}
This paper is concerned with the Minkowski convolution of viscosity solutions of fully nonlinear parabolic equations.  We adopt this convolution to compare viscosity solutions of initial-boundary value problems in different domains. As a consequence, we can for instance obtain parabolic power concavity of solutions to a general class of parabolic equations. Our results are applicable to the Pucci operator, the normalized $q$-Laplacians with $1<q\leq \infty$, the Finsler Laplacian, and more general quasilinear operators. 
\end{abstract}

\subjclass[2010]{
35D40. 
35K20, 
52A01. 
}
\keywords{power concavity, Minkowski addition, viscosity solutions, initial-boundary value problem}

\maketitle

\section{Introduction}

\subsection{Background and motivation}

This paper is connected to a general theory devised for the elliptic case in \cite{Salani1} and extended to the parabolic framework by two of the authors. 
In particular, we extend the results in \cite{IshSa2} and \cite{IshSa3} to a general class of fully nonlinear parabolic equations in the framework of viscosity solutions. In connection with the general theory of \cite{Salani1} and with the results and techniques of this paper, we also address the reader to the twin paper \cite{ILS2}, where we consider spatial concavity properties as well as Brunn-Minkowski type inequalities for parabolic and elliptic problems. 

Let us first describe the basic setting of our problem and introduce its background. 

Let $m\geq 2$ and $n\geq 1$. For any $i=1, 2, \ldots, m$, let $\Omega_i$ be a bounded smooth domain in $\R^n$. Let $\nu_i$ denote the inward unit normal vector to $\pO_i$.  For any 
\[
\lambda\in \Lambda_m=\left\{(\lambda_1, \ldots, \lambda_m)\in (0, 1)^m: \sum_{i=1}^m\lambda_i=1\right\},
\]
let $\Omega_\lambda$ be the Minkowski combination of $\Omega_i$, defined by 
\begin{equation}\label{minkowski domain}
\Omega_\lambda=\sum_{i=1}^m \lambda_i\Omega_i=\left\{\sum_{i=1}^m \lambda_i x_i: x_i\in \Omega_i,  i=1, 2, \ldots, m\right\}.
\end{equation}
It is easy to see that $\Omega_\lambda$ is bounded in $\R^n$.  
Notice that when $\Omega_i=\Omega$ for $i=1,\dots,m$, we have of course $\Omega\subseteq\Omega_\lambda$, but the inclusion is in general strict unless 
$\Omega$ is convex. Hereafter for simplicity we set $Q_i=\Omega_i\times (0, \infty)$ and $\partial Q_i=\left(\pO_i\times (0, \infty)\right)\cup \left(\Oba_i\times \{0\}\right)$ for $i=\lambda, 1, \ldots, m$. 
Our first aim is to connect the solution $u_\lambda$ of some Cauchy-Dirichlet problem in $\Omega_\lambda$ to the solutions $u_1,\dots,u_m$ of similar (but not necessarily the same) Cauchy-Dirichlet problems in $\Omega_1,\dots,\Omega_m$.

In particular, for $i=\lambda$ and $i=1, 2, \ldots, m$,
let us consider the following fully nonlinear Cauchy-Dirichlet problems:
\begin{numcases}{}
\partial_t u+F_i(x, t, u, \nabla u, \nabla^2 u)=0 &\text{in $Q_i$, }\label{general eq}\\
u=0 &\text{on $\partial Q_i$,}\label{bdry}
\end{numcases}
where $F_i: \overline{Q}_i\times [0, \infty)\times (\R^n\setminus\{0\})\times \S^n\to \R$ for $i=\lambda, 1, 2, \ldots, m$ are given continuous elliptic operators, with $F_\lambda$ suitably related to $F_1,\dots,F_m$.
As we said, we are interested in finding some kind of relationships (which we will clarify later) between the solution of problem \eqref{general eq}--\eqref{bdry} with $i=\lambda$ and the solutions with $i=1,\dots,m$.

Let $u_i$ be a positive solution of \eqref{general eq}--\eqref{bdry} in $Q_i$ for every $i=1, 2, \ldots, m$.  Let $1/2\leq \alpha\leq 1$ and $p<1$ be two given parameters and define the {\em $\alpha$-parabolic Minkowski $p$-convolution} of $\{u_i\}_{i=1}^m$ for any $\lambda\in \Lambda_m$ as follows:
\begin{equation}\label{envelope0}
\begin{aligned}
U_{p,\lambda}(x, t):=\sup\Bigg\{M_p  \left(u_1(x_1, t_1),\dots, u_m(x_m,t_m); \lambda\right) &:   (x_i, t_i)\in \overline{Q_i}, \\
&\hspace{-0.1cm} x=\sum_i \lambda_i x_i,\  t=\left(\sum_i \lambda_i t_i^\alpha\right)^{1\over \alpha}\Bigg\}.
\end{aligned}
\end{equation}
Here, for given $\lambda \in\Lambda_m$ and $p\in [-\infty,\,+\infty]$, 
$M_p(a_1,\dots,a_m;\lambda)$ denotes the usual weighted $p$-means (with weight $\lambda$) of $a=(a_1,\dots,a_m)\in[0,\infty)^m$, whose precise definition is given later in \eqref{pmeandef}.

As shown in \cite{IshSa3},  when the equations are semilinear with $F_i$  of the form
\begin{equation}\label{semilinear}
F_i(x, t, r, \xi, X)=-\tr X-f_i(x, t, r, \xi), \quad i=\lambda, 1, \ldots, m,
\end{equation}
then, under suitable assumptions on the behavior of the $u_i$'s on $\partial Q_i$'s, $U_{p,\lambda}$ is a subsolution of \eqref{general eq}--\eqref{bdry} with $i=\lambda$, provided that $f_\lambda$ and $\{f_i\}_{i=1}^m$ satisfy
\begin{equation}\label{nonlinearity cond0}
g_\lambda\left(\sum_i \lambda_i x_i, \sum_i \lambda_i t_i, \sum_i \lambda r_i, \xi\right)\geq \sum_{i=1}^m \lambda_i g_i (x_i, t_i, r_i, \xi)
\end{equation}
for any fixed $\xi\in \R^n$ and any $(x_i, t_i, r_i)\in Q_i\times (0, \infty)$, where
\begin{equation}\label{nonlinearity tran}
g_i(x, t, r, \xi)=r^{3-{1\over p}} f_i\left(x, t^{1\over \alpha}, r^{1\over p}, {1\over p}r^{{1\over p}-1}\xi\right), \quad i=\lambda, 1, \ldots, m.
\end{equation}
This, coupled with a comparison principle for \eqref{general eq},  results in a comparison between the solution of the problem in $\Omega_\lambda$ with the solutions in the $\Omega_i$'s, $i=1,\dots,m$, which consists in a sort of concavity principle for the solutions of the involved problems with respect to the Minkowski combination of the underlying domains. When the domains $\Omega_1,\dots,\Omega_m$ differ from each other, interesting applications are Brunn-Minkowski type inequalities for possibly connected functionals. 
For this, we refer to \cite{Salani1} and to the bibliography therein for the elliptic case and to \cite{IshSa3} for the parabolic case.

Notice that the condition \eqref{nonlinearity cond0} can be interpreted as a comparison relation between $f_\lambda$ and a certain type of concave combination of the $f_i$'s $(i=1, 2, \ldots, m)$ under the transformation \eqref{nonlinearity tran}.

When all the $\Omega_i$'s coincide with a convex domain $\Omega$ and all $f_i$ are the same for $i=\lambda, 1, \ldots, m$, all the problems clearly reduce to a single one. 
Then the above result, combined with a comparison principle for \eqref{general eq}--\eqref{bdry}, immediately implies that the unique solution $u$ of such an equation is $\alpha$-parabolically $p$-concave in the sense that
\begin{equation}\label{concavity}
u\left(\sum_i \lambda_i x_i, M_\alpha(t_1,\dots,t_m;\lambda)\right)\geq M_p\left(u_1(x_1, t_1)\dots,u_m(x_m,t_m);\lambda\right).
\end{equation}
This type of concavity results was established in \cite{IshSa2} and \cite{IshSa3} (see also \cite{IshSa1, IshSa2.1}). 
Note that \eqref{nonlinearity cond0} then turns into a concavity assumption for $g_\lambda$. 

When the $\Omega_i$'s truly differ from each other, then our result can be used to obtain Brunn-Minkowski type inequalities for related functionals, as it will be more explicitly described in \cite{ILS2} and has been already done in \cite{IshSa3} in the parabolic framework and similarly, suitably treating different specific cases, in \cite{CoSa, Co, CoCuSa, CoCu, salanima, LiuMaXu, Salani, BiSa} in the elliptic case. Notice that a general theory (for elliptic problems) is developed in \cite{Salani1}, where however only classical solutions and convex domains were considered, although all the results therein did not really need convexity of the involved domains. And indeed non convex domains have been explicitly treated in \cite{IshSa3}.

The purpose of this paper is to extend the results described above to a more general setting. Our generalization lies at the following three aspects. First, we study the problem for a general class of fully nonlinear parabolic equations, which certainly includes the known semilinear case. We even allow the equations to bear mild singularity caused by vanishing gradient. By ``mild singularity'' we mean that for each $i=\lambda, 1, \ldots, m$, 
there exists a continuous function $h_i: \ol{Q}_i\times [0, \infty)\to \R$ such that
\begin{equation}\label{f0}
h_i(x, t, r)=(F_i)_\ast(x, t, r, 0, 0)=(F_i)^\ast(x, t, r, 0, 0)\quad \text{for $(x, t, r)\in \ol{Q}_i\times [0, \infty)$},
\end{equation}
where $(F_i)_\ast$ and $(F_i)^\ast$ respectively stand for the lower and upper semicontinuous envelopes of $F_i$.
Our results are applicable to several important types of nonlinear operators including the Pucci operator, 
the normalized $q$-Laplacians $(1<q\leq \infty)$, and more general quasilinear operators. 

Secondly, in accordance to our generalization of the equations, another significant contribution of this paper is that we use the weaker notion of viscosity solutions rather than the classical solutions. We thus manage to reduce the $C^2$ regularity of the solutions in the main theorems of \cite{IshSa2, IshSa3}. Let us emphasize that it is indeed possible to investigate spatial convexity of solutions in the framework of viscosity theory; we refer to \cite{ES1, GGIS, ALL, Ju, LSZ} for viscosity techniques in different contexts and to \cite{Ke, Kor, KaBook, BianGu, IshSa0, IshSa0.5, INS1, INS2} etc for related results for classical solutions.
Our current work provides new results on parabolic power concavity of viscosity solutions, 
which are not considered in the aforementioned references  (but let us point out that, right after completing this work, we have learnt also about \cite{CraFra}, where viscosity solutions have been now considered to study Brunn-Minkowski type inequalities for the eigenvalues of fully nonlinear homogeneous elliptic operators).

Third, we allow more freedom to the parameters $\alpha$ and $p$, so that, depending on the involved operators, we can consider $\alpha\in[0,1]$ and $p\in(-\infty,1]$. Notice that, although there is no special difficulty, negative power concavity properties have not been explicitly treated before to our knowledge.

Throughout this paper we assume the following fundamental well-posedness results for any $i=\lambda, 1, \ldots, m$. 
\begin{itemize}
\item There exists a unique viscosity solution, locally Lipschitz in space,  to \eqref{general eq}--\eqref{bdry}.
\item The comparison principle holds for \eqref{general eq}--\eqref{bdry}, at least for $i=\lambda$; that is, if $u_\lambda$ and $v_\lambda$ are respectively an upper semicontinuous subsolution and a lower semicontinuous supersolution satisfying $u_\lambda\leq v_\lambda$ on $\partial Q_\lambda$, then $u_\lambda\leq v_\lambda$ in $\overline{Q}_\lambda$.  
\end{itemize}
We refer to \cite{CIL} and \cite{Gbook} for existence and uniqueness of viscosity solutions of \eqref{general eq}--\eqref{bdry}. For the reader's convenience, in Appendix (Section \ref{sec:app1}), we list more precise structure assumptions on the $F_\lambda$ besides \eqref{f0}, 
which guarantee the comparison principle; see more details also in \cite[Theorem 8.2]{CIL} and \cite[Theorem 3.6.1]{Gbook}. On the other hand, showing local Lipschitz regularity of the unique solution requires extra work and further assumptions on $F_i$. We refer to the extensive literature on this subject in the context of viscosity solutions, for example \cite{Ba0, Wang1, Wang2, KK2, BaSo, Ba3, Att} 
and references therein.

\subsection{Assumptions and main result}
Our main result is based on a condition connecting $F_\lambda$ and $F_i$ ($i=1, 2, \ldots, m$), which generalizes \eqref{nonlinearity cond0} in the fully nonlinear setting.  In order to give a clear view of this condition, we introduce the following transformed  operators with a parameter $k\in \R$.  Given $p<1$ and $\alpha\in(0,1)$, let $G_{i, k}^{p,\alpha}: \ol{Q}_i\times [0, \infty)\times (\R^n\setminus \{0\})\times \S^n\to\R$ be defined as follows for every $i=\lambda, 1, \ldots, m$:
\begin{equation}\label{operator-k}
\begin{array}{ll}
G_{i, k}^{p,\alpha}(x, t, r, \xi, X)=r^{k}  F_i\bigg(x,\ t^{1\over \alpha},\ r^{1\over p},\ {1\over p}r^{{1\over p}-1}\xi,\ {1\over p} r^{{1\over p}-1} X +{1-p\over p^2} r^{{1\over p}-2}   \xi\otimes \xi \bigg)&\text{if }p\neq 0\,,\\
G_{i, k}^{0,\alpha}(x, t, r, \xi, X)=e^{kr} F_i\bigg(x,\ t^{1\over \alpha},\ e^r,\ e^r\xi,\ e^r(X +  \xi\otimes \xi )\bigg)&\text{if }p= 0\,,
\end{array}
\end{equation}
for all $x, t, r, \xi, X\in \ol{Q}_i\times (0, \infty)\times (\R^n\setminus \{0\})\times \S^n$.

To apply our method, we need to find $k\in \R$ satisfying the following two key assumptions $\text{(H1)}$ and $\text{(H2)}$.
\begin{enumerate}
\item[$\text{(H1)}$] If $p\neq 0$, the parameter $k\in\R$ 
satisfies
\begin{equation}\label{par cond}
\text{ either } \quad {1\over p}-1+k\leq 0\quad \text{or}\quad \alpha\left({1\over p}-1+k\right)\geq 1.
\end{equation}

\item[$\text{(H2)}$] For any $\lambda\in \Lambda$ and any $(x, t)\in Q_\lambda$, $r\geq 0$, $\xi\in \R^n\setminus\{0\}$, and $Y\in \S^n$, 
\begin{equation}\label{key3}
G_{\lambda, k}^{p,\alpha}(x, t, r, \xi, Y) \leq \sum_{i=1}^m \lambda_i G_{i, k}^{p,\alpha}(x_i, t_i, r_i, \xi, X_i)
\end{equation}
holds for  $(x_i, t_i, r_i, X_i)\in Q_i\times (0, \infty)\times \S^n$ ($i=1, 2, \ldots, m$) satisfying 
\begin{equation}\label{vertex}
\sum_i \lambda_i x_i = x, \quad \sum_i \lambda_i t_i=t,  \quad \sum_i \lambda_i r_i=r,
\end{equation}
and 
\begin{equation}\label{key2}
\sgn^\ast(p) \begin{pmatrix}
\lambda_1 X_1 & & & \\
& \lambda_2 X_2 & & \\
& & \ddots & \\
& & & \lambda_{m} X_{m}
\end{pmatrix}
\leq 
\sgn^\ast(p)\begin{pmatrix}
\lambda_1^2 Y & \lambda_1\lambda_2 Y & \cdots & \lambda_1\lambda_{m} Y\\
\lambda_2 \lambda_1 Y & \lambda_2^2 Y & \cdots & \lambda_2 \lambda_{m} Y\\
\vdots & \vdots & \ddots & \vdots\\
\lambda_{m}\lambda_1 Y & \lambda_{m} \lambda_2 Y &  \cdots & \lambda_{m}^2 Y
\end{pmatrix},
\end{equation}
\end{enumerate}
where $\sgn^\ast(p)=1$ if $p\ge 0$ and $\sgn^\ast(p)=-1$ if $p<0$. 

We emphasize that when $p=0$, condition \eqref{par cond} can be removed, i.e. we can take any $k\in\R$. The condition $\text{(H1)}$ is equivalent to requiring the function $g_k(r, t)=r^{{1\over p}-1-k} t^{1-{1\over \alpha}}$ (when $p\neq 0$, while $g_k(r,t)=e^{(k+1)r}t^{1-{1\over \alpha}}$ when $p=0$)
to be convex in $(0, \infty)^2$.

The reason for us to impose $\text{(H2)}$ in the form involving $G_{i, k}^{p,\alpha}$ rather than $F_i$ is that we will later transform our equation \eqref{general eq} into another form, which is more compatible with our convexity argument. The operator $G_{i, k}^{p,\alpha}$ appears in the new equation. The  term $\sgn^\ast(p)$ is needed in \eqref{key2}, since for the transformed equation we will consider subsolutions when  $p\geq 0$ but supersolutions when $p<0$. 

Before stating our main result, we set 
\begin{equation}\label{par normal}
\tilde{\nu}_i(x):= \begin{cases} \nu_i(x) & \text{if $x\in \partial \Omega_i$,} \\
0 &\text{if $x\in \Omega_i$,}\end{cases}
\quad \text{and}\quad
\mu(t):=\begin{cases} 1 & \text{if $t=0$}, \\
0 & \text{if $t>0$.}
\end{cases}
\end{equation}

\begin{thm}[Subsolution property of Minkowski convolution]\label{thm main}
Fix $\lambda\in \Lambda_m$.  Assume that $\Omega_i$ is a bounded smooth domain in $\R^n$ for any $i=1, 2, \ldots, m$. Let $\Omega_\lambda$ be the Minkowski combination of $\{\Omega_i\}_{i=1}^m$ as defined in \eqref{minkowski domain}. Let $0< \alpha\leq 1$ and $p<1$. Suppose that there exists $k\in \R$ such that $\text{{\rm (H1)}}$ and ${\rm\text{{\rm(H2)}}}$ hold. 
 Let $u_i$ be the unique solution of \eqref{general eq}--\eqref{bdry} that is positive and locally Lipschitz in space in $Q_i$ for $i=1, 2, \ldots, m$. Assume in addition that for any $i=1, 2, \ldots, m$,
\begin{enumerate}
\item[(i)] $u_i$ is monotone in time, i.e., 
\begin{equation}\label{time monotone2}
u_i(x, t)\geq u_i(x, s)\quad \text{for any $x\in \Omega_i$ and $t\geq s\geq 0$};
\end{equation}
\item[(ii)] if $0<p\leq1$, then 
\begin{equation}\label{growth behavior}
{1\over \rho} u_i^p\left(x+\tilde{\nu}_i(x)\rho, t+\mu(t)\rho^{1/\alpha}\right)\to \infty \quad \text{as $\rho\to 0+$}
\end{equation}
for any $(x, t)\in \partial Q_i$.
\end{enumerate}
Then $U_{p,\lambda}$ as in \eqref{envelope0} is a subsolution of \eqref{general eq}--\eqref{bdry} with $i=\lambda$. 
\end{thm}

We can use our general result to cover \cite[Theorem 3.2]{IshSa3}. Indeed, if $p\neq 0$, 
by taking $k=3-1/p$ we get
\begin{equation}\label{heat op tran}
G_{i, k}^{p,\alpha}(x, t, r, \xi, X)=-{1\over p}r^{2} \tr X-{1-p\over p^2}r-g_i(x, t, r, \xi)
\end{equation}
for all $(x, t, r, \xi, X)\in \Omega_\lambda$. We can verify the assumption $\text{(H2)}$ in Theorem \ref{thm main} holds with the choice $k=3-1/p$ and the condition
\eqref{nonlinearity cond0}.  In the case $p=0$, we can choose $k=1$ to show that the same result holds 
under condition \eqref{nonlinearity cond0} but with 
\begin{equation}\label{nonlinearity tran0}
g_i(x, t, r, \xi)=e^r f_i\left(x, t^{1\over \alpha}, e^r, e^r\xi\right), \quad i=\lambda, 1, \ldots, m.
\end{equation}
See more details in Section \ref{sec:ex1}.

Compared to the key conditions $\text{(H1)}$ and $\text{(H2)}$, the additional assumptions (i)--(ii) are more technical. 
Notice however that assumption \eqref{growth behavior} is not needed for $p\leq 0$. Moreover, for $p\in(0,1)$, even if in applications $F_i$  may not fulfill (i)--(ii), we can fix the issue by perturbing $F_i$ with a small $\vep>0$ as
\begin{equation}\label{perturb F}
F_{i, \vep}=F_i-\vep\quad (i=1, 2, \ldots, m);
\end{equation}
in other words, we instead consider the equation 
\begin{equation}\label{perturb eq}
\partial_t u+F_{i, \vep}(x, t, u, \nabla u, \nabla^2 u)=0 \quad \text{in $Q_i$. }
\end{equation} 
It turns out that such perturbation meet our needs in most of our applications. 
For $p\in(0,1)$, we can prove (i) and (ii) for \eqref{perturb eq} with a larger class of parabolic operators $F_i$; see Appendix \ref{sec:app2} and Appendix \ref{sec:app3} for clarification.  
Such a perturbation causes no harm to the applications of our main results, since all of the other assumptions continue to hold in Theorem \ref{thm main} with $F_i$ replaced by $F_{i, \vep}$. We can still obtain the desired results by first considering the approximate problem \eqref{perturb eq} and then passing to the limit as $\vep\to 0$ by standard stability theory. Let us finally notice that, although Theorem \ref{thm main} holds the same also for $p=1$, in this case it is very hard to get assumption \eqref{growth behavior}, which would require $u_i$ to have vertical slope on the boundary (and indeed it is very hard to have concave solutions).

Our proof of Theorem \ref{thm main} is based on the following two steps. We first take 
\begin{equation}\label{change variable}
v_i(x, t)=\left\{\begin{array}{ll}u_i^p\left(x, t^{1\over \alpha}\right)&\,\,\text{ if }p\neq 0,\\
\log u\left(x, t^{1\over \alpha}\right)&\,\,\text{ if }p= 0,
\end{array}\right.
\end{equation}
for all $i=\lambda, 1, \ldots, m$. 
It is not difficult to verify, at least formally, that $v_i$ solves
\begin{equation}\label{general eq2}
{\alpha\over p}v_i^{{1\over p}-1}t^{1-{1\over \alpha}}\partial_t v_i +F_i\bigg(x,\ t^{1\over \alpha},\ v_i^{1\over p},\ {1\over p}v_i^{{1\over p}-1}\nabla v_i,\ {1\over p} v_i^{{1\over p}-1} \nabla^2 v_i  +{1-p\over p^2}   v_i^{{1\over p}-2}   \nabla v_i\otimes \nabla v_i\bigg)=0 
\end{equation}
if $p\neq 0$ and 
\begin{equation}\label{general eq2-0}
e^{v_i}t^{1-{1\over \alpha}}\partial_t v_i +F_i\bigg(x,\ t^{1\over \alpha},\ e^{v_i},\ e^{v_i}\nabla v_i,\ e^{v_i} \nabla^2 v_i  +e^{v_i}   \nabla v_i\otimes \nabla v_i\bigg)=0 
\end{equation}
if $p=0$, which are respectively equivalent to 
\begin{equation}
\label{general eq3}
\begin{split}
 & v_i^{{1\over p}-1+k} t^{1-{1\over \alpha}} \partial_t v_i(x, t)+{p\over \alpha}G_{i, k}^{p,\alpha}\left(x, t, v_i(x, t), \nabla v_i(x, t), \nabla^2 v_i(x, t)\right)=0,\\
 & e^{(k+1)v_i} t^{1-{1\over \alpha}} \partial_t v_i(x, t)+{1\over \alpha}G_{i, k}^{0,\alpha}\left(x, t, v_i(x, t), \nabla v_i(x, t), \nabla^2 v_i(x, t)\right)=0,
\end{split}
\end{equation}
for any given parameter $k\in \R$. Here $G_{i, k}^{p,\alpha}$ is given by \eqref{operator-k}. In Section \ref{sec:transformation}, we rigorously show that  $u_i$ is a viscosity subsolution of \eqref{general eq} if and only if 
$v_i$ is a viscosity subsolution (resp., supersolution) of \eqref{general eq3} when $p \geq 0$ (resp., $p<0$).

After such a transformation,  we next take the Minkowski convolution of $v_i$'s  as follows: 
\begin{equation}\label{envelope2}
\begin{split}
 & V_{p,\lambda}(x, \tau)\\
 & :=\left\{\begin{array}{ll}
\sup\left\{\displaystyle{\sum_{i=1}^m} \lambda_i v_i(x_i, \tau_i): \ (x_i, \tau_i)\in \overline{Q_i}, \  
x=\sum_{i=1}^m \lambda_i x_i, \ \tau=\sum_{i=1}^m \lambda_i \tau_i\right\}&\text{if }p\geq 0,\\
\\
\inf\left\{\displaystyle{\sum_{i=1}^m} \lambda_i v_i(x_i, \tau_i): \ (x_i, \tau_i)\in \overline{Q_i}, \  
x=\sum_{i=1}^m \lambda_i x_i, \ \tau=\sum_{i=1}^m \lambda_i \tau_i\right\}&\text{if }p< 0,
\end{array}\right.
\end{split}
\end{equation}
for every $(x, \tau)\in \ol{Q}_\lambda$. It is clear that 
\[
V_{p,\lambda}(x, \tau)=\left\{\begin{array}{ll}U_{p,\lambda}\left(x, \tau^{1\over \alpha}\right)^p&\,\,\text{if }p\neq 0,\\
\log U_{p,\lambda}\left(x, \tau^{1\over \alpha}\right)&\,\,\text{if }p= 0.
\end{array}\right.
\]
To prove Theorem \ref{thm main}, it thus suffices to prove that $V_{p,\lambda}$ is a subsolution if $p\geq0$ or a supersolution if $p<0$ of \eqref{general eq3} with $i=\lambda$.
The rest of the proof is inspired by \cite{ALL}, where the supersolution property is studied for the convex envelope of viscosity solutions to fully nonlinear elliptic equations with state constraint or Dirichlet boundary conditions. The key is to establish a relation between the semijets (weak derivatives) of $v_i$ and $V_{p,\lambda}$, which combined with $\text{(H1)}$--$\text{(H2)}$, leads to the desired conclusion.


\subsection{Applications to parabolic power concavity}
We can use Theorem \ref{thm main} to study the parabolic power concavity of viscosity solutions to a general class of fully nonlinear parabolic equations. More precisely, when $F_i=F_\lambda$ and $\Omega_i=\Omega_\lambda$  for all $i=1, 2, \ldots, m$ with $m=n+2$, assuming the convexity of $\Omega_\lambda$, we can apply the above result to the unique solution $u$ of  \eqref{general eq}--\eqref{bdry} with $i=\lambda$ to deduce that 
\begin{equation}\label{envelope}
\begin{split}
 & u^\star(x, t)\\
 & :=\sup\left\{ \left(\sum_{i=1}^m \lambda_i u^p(x_i, t_i)\right)^{1\over p}: \ (x_i, t_i)\in \overline{Q_i}, \  x=\sum_{i=1}^m \lambda_i x_i, \  
t=\left(\sum_{i=1}^m \lambda_i t_i^\alpha\right)^{1\over \alpha}\right\}
\end{split}
\end{equation}
is a subsolution of \eqref{general eq}--\eqref{bdry} with $i=\lambda$. Since $u\leq u^\star$ by the definition and the comparison principle implies that $u\geq u^\star$ in $\ol{Q}_\lambda$, we obtain $u=u^\star$, i.e. the parabolic power concavity of $u$ in the sense of \eqref{concavity}.
 In this case, the assumption $\text{(H2)}$ becomes the following convexity assumption on the operator $G_{\lambda,k}^{p,\alpha}$ defined by \eqref{operator-k}:
\begin{enumerate}
\item[(H2a)] For any $\lambda\in \Lambda$ and any $(x, t)\in Q$, $r\geq 0$, $\xi\in \R^n\setminus\{0\}$, and $Y\in \S^n$,
\[
G_{\lambda,k}^{p,\alpha}(x, t, r, \xi, Y)\leq \sum_{i=1}^{n+2} \lambda_i G_{\lambda,k}^{p,\alpha}(x_i, t_i, r_i, \xi, X_i)
\]
holds whenever $(x_i, t_i)\in Q_i, r_i>0$ and $X_i \in \S^n$ fulfilling \eqref{vertex} and \eqref{key2} with $m=n+2$.
\end{enumerate}

\begin{thm}[Parabolic power concavity]\label{thm main1}
Assume that $\Omega_\lambda\subset \R^n$ is a smooth bounded convex domain and that $F_\lambda$ satisfies \eqref{f0} with $i=\lambda$. Let  $u$ be the unique viscosity solution of \eqref{general eq}--\eqref{bdry} with $i=\lambda$ 
{\rm({\it that is positive and locally Lipschitz in space in $Q_\lambda=\Omega_\lambda\times (0, \infty)$})}. 
Let $k\in \R$, $0< \alpha\leq 1$, and $p\leq1$. Assume that {\rm(H1)} and {\rm (H2a)} hold, and, in addition, that
\begin{enumerate}
\item[(i)] $u$ is monotone in time, i.e., 
\[
u(x, t)\geq u(x, s)\quad \text{for any $x\in \Omega_\lambda$ and $t\geq s\geq 0$};
\]
\item[(ii)] if $p>0$, then 
\[
{1\over \rho} u^p\left(x+\tilde{\nu}_0(x)\rho, t+\mu(t)\rho^{1/\alpha}\right)\to \infty \quad \text{as $\rho\to 0+$}
\]
for any $(x, t)\in \ol{Q}_i$.
\end{enumerate}
   Then $u$ is $\alpha$-parabolically $p$-concave in $Q_\lambda$ in the sense of \eqref{concavity}. 
\end{thm}
It is worth remarking that (H2a) is actually slightly weaker than the usual convexity of $(x, t, r, X)\mapsto G_{\lambda,k}^{p,\alpha}(x, t, r, \xi, X)$ combined with the ellipticity of $F_\lambda$, since \eqref{key2} implies that 
\[
\sgn^\ast(p)\sum_{i} \lambda_i X_i\leq \sgn^\ast(p)Y.
\]

As Theorem \ref{thm main}, Theorem \ref{thm main1} generalizes some previous results, precisely \cite[Theorem 3]{IshSa2} and \cite[Theorem 4.2]{IshSa3}, which treat in the special case
\[
F_i(x, t, r, \xi, X)=-\tr X-f(x, t, r, \xi)\quad \text{ ($i=\lambda, 1, 2, \ldots, m$) }
\]
with $f\geq 0$ a given continuous function such that
\begin{equation}\label{concave nonlinearity}
(x, t, r)\mapsto\left\{ \begin{array}{ll}
r^{3-{1\over p}} f\left(x, t^{1\over \alpha}, r^{1\over p}, {1\over p}r^{{1\over p}-1} \xi\right)&\,\,\text{if }p\neq 0,\\
e^rf(x,t^{1/\alpha},e^r,e^r\xi)&\,\,\text{if }p=0,
\end{array}\right.
\end{equation}
is concave in $Q_\lambda\times (0, \infty)$ for any $\xi\in \R^n$.

For most applications of Theorem \ref{thm main} and Theorem \ref{thm main1}, we can take $k=3-{1/p}$ for $p\neq 0$. It is clear that $\text{(H1)}$ is satisfied in this case. Denoting
\begin{equation}\label{gbar}
\ol{G}=G_{\lambda, 3-{1/p}}^{p,\alpha} \qquad (p\neq 0),
\end{equation}
we see that the equation \eqref{general eq3} with $i=\lambda$ reduces to 
\begin{equation}\label{general eq2s}
v^{2} t^{1-{1\over \alpha}} \partial_t v+{p\over \alpha}\ol{G}\left(x, t, v, \nabla v, \nabla^2 v\right)=0.
\end{equation}
To meet the requirement (H2a) in Theorem \ref{thm main1}, we only need to assume the following. 
\begin{enumerate}
\item[(H2b)] For any $\lambda\in \Lambda$ and any $\xi\in \R^n\setminus\{0\}$, 
\begin{equation}\label{key1}
\sum_i \lambda_i \ol{G}(x_i, t_i, r_i, \xi, X_i)\geq \ol{G}\left(\sum_i \lambda_i x_i,  \sum_i \lambda_i t_i, \sum_i \lambda_i r_i, \xi, Y\right)
\end{equation}
for all $(x_i, t_i)\in Q_\lambda, r_i>0$, and $X_i, Y\in \S^n$ satisfying \eqref{key2} with $m=n+2$. 
\end{enumerate}

\begin{cor}[A special case for parabolic power concavity]\label{cor main2}
Let $\Omega\subset \R^n$ be a bounded smooth convex domain. Assume that $F_\lambda$ satisfies \eqref{f0} with $i=\lambda$. 
Let $0<\alpha\leq 1$ and $0\neq p\leq1$. Assume that {\rm (H2b)} holds. 
Let $u$ be a unique viscosity solution of \eqref{general eq}--\eqref{bdry} with $i=\lambda$ 
{\rm ({\it that is positive and locally Lipschitz in space in $Q_\lambda$})}. Assume in addition that $u$ satisfies {\rm(i)} and {\rm(ii)} 
in Theorem~{\rm\ref{thm main1}}. 
 Then $u$ is $\alpha$-parabolically $p$-concave in $Q_\lambda$.
\end{cor}
We can verify that (H2b) holds when the operator $F_\lambda$ is in the form 
\[
F_\lambda(x, t, r, \xi, X)=\L(\xi, X)-f(x, t, r, \xi),
\]
where $\L$ is a degenerate elliptic operator satisfying proper assumptions (for instance $\L$ is 1-homogeneous with respect to $X$ and $0$-homogeneous with respect to $\xi$) and $f\geq 0$ is a continuous function such that \eqref{concave nonlinearity} is concave for any fixed $\xi\in \R^n$. Examples of $\L$ include the Laplacian, the Pucci operator, the normalized $q$-Laplacian ($1<q\leq \infty$), the Finsler Laplacian, etc.; see details in Section \ref{sec:ex}.
\vspace{8pt}

\noindent
{\bf Acknowledgments.}
Part of this research was completed while the first and second authors were visiting the third author in October 2018 at Universit\`{a} di Firenze, whose hospitality is gratefully acknowledged. 

The work of the first author was partially supported by 
the Grant-in-Aid for Scientific Research (A) (No. 15H02058)
from JSPS (Japan Society for the Promotion of Science).
The work of the second author was partially supported by Grant-in-Aid for Young Scientists (No. 16K17635) from JSPS and by the Grant from Central Research Institute of Fukuoka University (No. 177102).
The work of the third author was partially  supported by INdAM through a GNAMPA project.

\section{Preliminaries}

\subsection{Power means of nonnegative numbers}
For $a=(a_1,\dots,a_m)\in(0,\infty)^m$, $\lambda \in\Lambda_m$, and $p\in [-\infty,\,+\infty]$, 
we set 
\begin{equation}
\label{pmeandef}
 M_p(a;\lambda):=
\left\{
\begin{array}{ll}
\left[\lambda_1a_1^p+\lambda_2a_2^p+\cdots+\lambda_ma_m^p\right]^{1/p}
& \mbox{if}\quad p\neq -\infty,\,0,\,+\infty,\vspace{5pt}\\ 
\max\{a_1,\dots,a_m\} & \mbox{if}\quad p=+\infty,\vspace{5pt}\\
a_1^{\lambda_1}\cdots a_m^{\lambda_m} & \mbox{if}\quad p=0,\vspace{5pt}\\
\min\{a_1,a_2,\dots,a_m\} & \mbox{if}\quad p=-\infty,
\end{array}
\right.
\end{equation}
which is the ($\lambda$-weighted) $p$\,-mean of $a$. 

For $a=(a_1,\dots,a_m)\in[0,\infty)^m$, we define $M_p(a;\lambda)$ as above if $p\geq0$ and 
$M_p(a;\lambda)=0$ if $p<0$ and $\prod_{i=1}^ma_i=0$. 

Notice that $M_p(a;\lambda)$ is a continuous function of the argument $a$.
Due to the Jensen inequality, we have 
\begin{equation}
\label{eq:2.2}
M_p(a;\lambda)\leq M_q(a;\lambda)\quad \mbox{if}\quad -\infty\le p\le q\le\infty,
\end{equation} 
for any $a\in[0,\infty)^m$ and $\lambda \in \Lambda_m$. 
Moreover, it easily follows that
$$
\lim_{p\rightarrow +\infty}
M_p(a;\lambda)=M_{+\infty}(a;\lambda),
\quad
\lim_{p\rightarrow 0}
M_p(a;\lambda)=M_0(a,\lambda),
\quad
\lim_{p\rightarrow -\infty}
M_p(a;\lambda)=M_{-\infty}(a;\lambda).
$$
For further details, see e.g. \cite{HLP}. 

\subsection{Definition of viscosity solutions}\label{sec:viscosity}

We recall the definition of viscosity solutions to \eqref{general eq}, which can also be found in \cite{CIL, Gbook}. 
In Appendix~\ref{sec:app}, we review more properties of viscosity solutions that are needed in this work. 

Let $\Omega$ be a bounded smooth domain in $\R^n$. Let $\O$ denote an arbitrary open subset of $Q=\Omega\times (0, \infty)$. Consider a general parabolic equation
\begin{equation}\label{single eq}
\partial_t u+F(x, t, u, \nabla u, \nabla^2 u)=0
\end{equation}
in $Q$, where $F$ is a proper elliptic operator.  

Here, by elliptic we mean that
\begin{equation}\label{eq elliptic}
F(x, t, r, \xi, X_1)\leq F(x, t, r, \xi, X_2)
\end{equation}
for all $(x, t, r, \xi)\in \Oba_i\times [0, \infty)\times [0, \infty)\times (\R^n\setminus\{0\})$ and $X_1, X_2\in \S^n$ satisfying $X_1\geq X_2$. We also recall that $F$ is proper if 
there exists $c\in \R$ such that 
\begin{equation}\label{eq monotone}
F(x, t, r_1, \xi, X)+cr_1\leq F(x, t, r_2, \xi, X)+cr_2
\end{equation}
for all $(x, t, \xi, X)\in \Oba_i\times [0, \infty)\times (\R^n\setminus\{0\})\times \S^n$ and $r_1, r_2\in [0, \infty)$ satisfying $r_1\leq r_2$.

We further assume that $F$ satisfies \eqref{f0} with the subindex $i$ omitted.

\begin{defi}\label{def1}
A locally bounded upper (resp., lower) semicontinuous function $u: \O\to \R$ is said to be a subsolution (resp., supersolution) of \eqref{single eq} in $\O$ if whenever there exist $(x_0, t_0)\in \O$ and $\phi\in C^2(\O)$ such that $u-\phi$ attains a maximum (resp., minimum) at $(x_0, t_0)$, we have 
\begin{equation*}
\begin{split}
 & \partial_t\phi(x_0, t_0)+ F_\ast(x_0, t_0, u(x_0, t_0), \nabla \phi(x_0, t_0), \nabla^2 \phi(x_0, t_0))\leq 0\\
 & \left(\text{resp., }\partial_t\phi(x_0, t_0)+ F^\ast(x_0, t_0, u(x_0, t_0), \nabla \phi(x_0, t_0), \nabla^2 \phi(x_0, t_0))\geq 0\right).
\end{split}
\end{equation*}
A continuous function $u: \O\to \R$ is called a solution of \eqref{single eq} in $\O$ if it is both a subsolution and a supersolution in $\O$. 
\end{defi}

It is clear that $F_\ast=F^\ast=F$ in $\O\times (0, \infty)\times \R^n\times \S^n$ provided that $F$ is assumed to be continuous in $\O\times (0, \infty)\times \R^n\times \S^n$. 

\begin{rmk}\label{rmk def1}
It is standard in the theory of viscosity solutions to use the semijets to give an equivalent definition. More precisely, for any $(x_0, t_0)\in \O$, setting 
$P^{2, +}u(x_0, t_0)\subset \R\times \R^n\times \S^n$ as 
\[
\begin{aligned}
P^{2, +} u(x_0, t_0)=\bigg\{(\tau, \xi, X): u(x, t)\leq & u(x_0, t_0) +\tau (t-t_0) +\la \xi, x-x_0 \ra\\
&+ {1\over 2}\la X (x-x_0), (x-x_0)\ra+o(|t-t_0|+|x-x_0|^2)\bigg\}
\end{aligned}
\]
and its ``closure'' as
\[
\begin{aligned}
\overline{P}^{2, +}u(x_0, t_0)=\bigg\{(\tau, \xi, X): \ & \text{ there exist $(x_j, t_j)\in \O$ and $(\tau_j, \xi_j, X_j)\in P^{2, +}u(x_j, t_j)$} \\
&\text{such that $(x_j, t_j, \tau_j, \xi_j, X_j)\to (x_0, t_0, \tau, \xi, X)$ as $j\to \infty$}
\bigg \},
\end{aligned}
\]
we then say $u$ is a subsolution of \eqref{single eq} if 
\[
\tau+ F_\ast(x_0, t_0, u(x_0, t_0), \xi, X)\leq 0
\]
for every $(\tau, \xi, X)\in \overline{P}^{2, +}u(x_0, t_0)$. The semijet $P^{2, -}u(x_0, t_0)$, 
its closure, and supersolutions can be analogously defined in a symmetric way.
\end{rmk}

If $F(x, t, r, \xi, X)$ is mildly singular at $\xi=0$, i.e. \eqref{f0} holds, one can use the following equivalent definition, called $\mathcal{F}$-solutions as in \cite{Gbook}. 
\begin{defi}\label{def2}
Suppose that there exists $h\in C(\Oba\times [0, \infty)\times [0, \infty))$ such that  
\[
h(x, t, r)=F^\ast(x, t, r, 0, 0)=F_\ast(x, t, r, 0, 0)
\]
holds for all $(x, t, r)\in \O\times \R$. A locally bounded upper (resp., lower) semicontinuous function $u: \O\to \R$ is said to be a subsolution (resp., supersolution) of \eqref{single eq} in $\O$ if, whenever there exist $(x_0, t_0)\in \O$ and $\phi\in C^2(\O)$ such that $u-\phi$ attains a maximum (resp., minimum) at $(x_0, t_0)$, we have 
\begin{equation*}
\begin{split}
 & \partial_t\phi(x_0, t_0)+ F(x_0, t_0, u(x_0, t_0), \nabla \phi(x_0, t_0), \nabla^2 \phi(x_0, t_0))\leq 0\\
 & \left(\text{resp., }\partial_t\phi(x_0, t_0)+ F(x_0, t_0, u(x_0, t_0), \nabla \phi(x_0, t_0), \nabla^2 \phi(x_0, t_0))\geq 0\right)
\end{split}
\end{equation*}
when $\nabla \phi(x_0, t_0)\neq 0$ and 
\begin{equation*}
\begin{split}
 & \partial_t\phi(x_0, t_0)+ h(x_0, t_0, u(x_0, t_0))\leq 0\\
 & \left(\text{resp., }\partial_t\phi(x_0, t_0)+ h(x_0, t_0, u(x_0, t_0))\geq 0\right)
\end{split}
\end{equation*}
when $\nabla \phi(x_0, t_0)=0$ and $\nabla^2\phi(x_0, t_0)=0$.
\end{defi}
\begin{rmk}\label{rmk def2}
In the definition of subsolutions by semijets, these conditions are written as follows: for any  $(\tau, \xi, X)\in \overline{P}^{2, +}u(x_0, t_0)$, we require that 
\begin{equation*}
\begin{array}{ll}
\tau+ F(x_0, t_0, u(x_0, t_0), \xi, X)\leq 0 & \quad\text{if}\quad \xi\neq 0,\vspace{5pt}\\ 
\tau+h(x_0, t_0, u(x_0, t_0))\leq 0 & \quad\mbox{if}\quad\xi=\lambda\quad\text{and}\quad X=0.
\end{array}
\end{equation*}
\end{rmk}

\section{A useful transformation of the unknown function}\label{sec:transformation}

A straightforward way to study this problem is to directly turn the unknown function into a form that fits the desired parabolic power concavity.

If $u_i$ is a smooth positive subsolution of \eqref{general eq} and $F$ is not mildly singular, then by direct calculations we see that $v_i$ defined in \eqref{change variable} is a smooth subsolution of  \eqref{general eq2} for all $i=\lambda, 1, \ldots, m$.
In fact, we have 
\begin{equation*}
\begin{split}
 & u_i(x, t)=v_i^{1\over p}(x, t^\alpha),
\quad
\partial_t u_i(x, t)={\alpha\over p}v_i^{{1\over p}-1}t^{\alpha-1}\partial_t v_i(x, t^\alpha),
\quad
\nabla u_i(x, t)= {1\over p}v_i^{{1\over p}-1}\nabla v_i(x, t^\alpha),\\
 & \nabla^2 u_i(x, t)={1\over p} v_i^{{1\over p}-1} \nabla^2 v_i (x, t^\alpha) +{1-p\over p^2}  v_i^{{1\over p}-2}   \nabla v_i(x, t^\alpha)\otimes \nabla v_i (x, t^\alpha).
\end{split}
\end{equation*}
Plugging these into \eqref{general eq}, we easily obtain \eqref{general eq2}. It is clear that positive smooth solutions of \eqref{general eq2} are equivalent to positive smooth solutions of \eqref{general eq3}, where  $G_{i, k}^{p,\alpha}$ is given by \eqref{operator-k}.

When $u_i$ is not necessarily smooth, we can interpret such a result in the viscosity sense. 
\begin{lem}[Sub/supersolution properties under transformation]\label{lem transform}
Fix $i=\lambda, 1, \ldots, m$ arbitrarily. Assume that \eqref{f0} holds.  Let $u_i$ be positive and upper semicontinuous in $Q_i$. Let $v_i$ be given by \eqref{change variable}. Then
\begin{itemize}
\item if $p\geq 0$, $u_i$ is a viscosity subsolution of \eqref{general eq} if and only if $v_i$ is a viscosity subsolution of \eqref{general eq3} in $Q_i$;\\
\item for $p<0$, $u_i$ is a viscosity subsolution of \eqref{general eq} if and only if $v_i$ is a viscosity supersolution of \eqref{general eq3} in $Q_i$.  
\end{itemize}
Moreover, a symmetric result holds also for supersolutions.
\end{lem}
\begin{proof}
Let us give the proof in details for the case $p> 0$ and $u_i$ is a subsolution of \eqref{general eq}, then let us prove that this  implies that $v_i$ is a subsolution of \eqref{general eq2}. The converse implication can be similarly shown. 

Assume that there exist $(x_0, t_0)\in Q_i$ and $\phi\in C^2(\ol{Q}_i)$ such that 
\[
\max_{\ol{Q}_i}(v_i-\phi)=(v_i-\phi)(x_0, t_0)=0.
\]
In other words, we have 
\[
v_i(x_0, t_0)=\phi(x_0, t_0), \quad v_i(x, t)\leq \phi(x, t)\quad \text{for all $(x, t)\in Q_i$.}
\]
Since $v_i>0$ in $Q_i$, it follows that 
\[
u_i(x_0, t_0^{1/\alpha})=\phi^{1/ p}\left(x_0, t_0\right), \quad u_i(x, t^{1/\alpha})\leq \phi^{1/p}(x, t)\quad \text{for all $(x, t)\in Q_i$.}
\]
This implies that $u_i(x, t)-\psi(x, t)$ attains a maximum over $Q_i$ at $(x_0, t_0^{1/\alpha})$, where
$\psi(x, t)=\phi^{1\over p}(x, t^\alpha)$. 

Suppose that $\nabla \phi(x_0, t_0)\neq 0$.  Then $\nabla \psi(x_0, t_0)\neq 0$. Since $u_i$ is a subsolution of \eqref{general eq}, we see that 
\[
\partial_t \psi+F_i\left(x_0, t_0^{1/ \alpha}, u_i, \nabla \psi, \nabla^2 \psi\right)\leq 0 \quad \text{at $(x_0, t_0^{1/ \alpha})$.}
\]
 By direct calculations it follows that at $(x_0, t_0)$
\begin{equation}\label{eq sub equiv1}
{\alpha\over p}v_i^{{1\over p}-1}t^{1-{1\over \alpha}}\partial_t \phi +F_i\bigg(x_0,\ t_0^{1\over \alpha},\ v_i^{1\over p},\ {1\over p}v_i^{{1\over p}-1}\nabla \phi,\ {1\over p} v_i^{{1\over p}-1} \nabla^2 \phi  +{1-p\over p^2}  v_i^{{1\over p}-2}   \nabla \phi\otimes \nabla \phi\bigg)\leq 0.
\end{equation}
Multiplying \eqref{eq sub equiv1} by $p v_i(x_0, t_0)^k$, we obtain 
\[
v_i(x_0, t_0)^{{1\over p}-1+k} t^{1-{1\over \alpha}} \partial_t\phi_i(x_0, t_0)+{p\over \alpha}G_{i, k}^{p,\alpha}(x_0, t_0, v_i(x_0, t_0), \nabla \phi_i(x_0, t_0), \nabla^2 v_i(x_0, t_0))\leq 0.
\]

If $\nabla \phi(x_0, t_0)=0$, we have $\nabla\psi(x_0, t_0^{1/\alpha})=0$. Using Definition \ref{def2}, we assume $\nabla^2 \phi(x_0, t_0)=0$, which is equivalent to $\nabla^2\psi\left(x_0, t_0^{1/\alpha}\right)=0$. 
We thus can apply the definition of subsolution on $u_i$ to obtain 
\[
\partial_t\psi(x_0, t_0^{1/\alpha})+h_i\left(x_0, t_0^{1/\alpha}, u_i(x_0, t_0^{1/\alpha})\right)\leq 0,
\]
which yields
\[
v_i(x_0, t_0)^{{1\over p}-1}t_0^{1-{1\over \alpha}}\partial_t \phi(x_0, t_0)+{p\over \alpha}\tilde{h}_i\left(x_0, t_0^{1/ \alpha}, v_i(x_0, t_0)\right)\leq 0. 
\]
The proof of the case $p> 0$  and $u_i$ is a subsolution is thus complete. The cases $p=0$ and $p<0$ can be treated similarly, and the same for the symmetric case when $u_i$ is a supersolution.
\end{proof}

If $F$ is mildly singular, it is not difficult to see that 
\begin{equation}\label{singularity limit1}
G_{i, k}^{p,\alpha}(x, t, r, \xi, X)\to r^k h_i\left(x, t^{1\over \alpha}, r^{1\over p}\right) \quad \text{as $\xi\to 0,\ X\to 0$}
\end{equation}
locally uniformly for all $(x, t, r)\in \ol{Q}_i\times [0, \infty)$ and all $i=\lambda, 1, \ldots, m$. In other words, the operator $G_{i, k}^{p,\alpha}$  satisfies the same properties as in \eqref{f0}.  We are thus able to apply Definition \ref{def2} to define the sub- and supersolutions of \eqref{general eq3}. Let us denote
\begin{equation}\label{operator-k0}
\tilde{h}_{i}(x, t, r)=r^k h_i\left(x, t^{1\over \alpha}, r^{1\over p}\right)
\end{equation}
for all $(x, t, r)\in \ol{Q}_i\times [0, \infty)$ and $i=\lambda, 1, \ldots, m$.

\begin{prop}
Assume that \eqref{f0} holds for each $i=\lambda, 1, \ldots, m$. Suppose that there exists $k\in \R$ such that 
{\rm (H2)} holds. Then $\tilde{h}_i$ given by \eqref{operator-k0} satisfies 
\begin{equation}\label{singularity convexity}
\sum_i \lambda_i \tilde{h}_i(x_i, t_i, r_i)\geq \tilde{h}_0\left(\sum_i \lambda_i x_i,  \sum_i \lambda_i t_i, \sum_i \lambda_i r_i\right)
\end{equation}
for any $\lambda\in \Lambda$, $(x_i, t_i)\in Q_i$, and $r_i>0$. 
\end{prop}
\begin{proof}
Since 
\[
G_{\lambda,k}^{p,\alpha}(x, t, r, \xi, 0)\to \tilde{h}_0(x, t, r)\quad \text{locally uniformly as $\xi\to 0$}
\]
for any $\vep>0$, there exists $\xi_\vep\in \R^n\setminus \{0\}$ such that 
\begin{equation}\label{singularity limit2}
G_{i, k}^{p,\alpha}\left(\sum_i \lambda_i x_i,  \sum_i \lambda_i t_i, \sum_i \lambda_i r_i, \xi_\vep, 0\right)\geq \tilde{h}_0\left(\sum_i \lambda_i x_i,  \sum_i \lambda_i t_i, \sum_i \lambda_i r_i\right)-\vep.
\end{equation}
Since \eqref{key2} clearly holds with $Y=X_i=0$ for all $i=1, 2, \ldots, ,n+2$, by $\text{(H2)}$ we get 
\[
\sum_i \lambda_i G_{i, k}^{p,\alpha}(x_i, t_i, r_i, \xi_\vep, 0)\geq G_{\lambda,k}^{p,\alpha}\left(\sum_i \lambda_i x_i,  \sum_i \lambda_i t_i, \sum_i \lambda_i r_i, \xi_\vep, 0\right).
\]
which by \eqref{singularity limit2} yields
\[
\sum_i \lambda_i G_{i, k}^{p,\alpha}(x_i, t_i, r_i, \xi_\vep, 0)\geq \tilde{h}_0\left(\sum_i \lambda_i x_i,  \sum_i \lambda_i t_i, \sum_i \lambda_i r_i\right)-\vep
\]
Sending $\vep\to 0$, we obtain \eqref{singularity convexity} by \eqref{singularity limit1} and \eqref{operator-k0}. 
\end{proof}

When $p>0$, we easily see that 
$v_i$ satisfies the same initial and boundary conditions as $u_i$. 
Therefore we can write the Cauchy-Dirichlet problem for $v_i$ $(i=\lambda, 1, \ldots, m)$ as
\begin{numcases}{}
v^{{1\over p}-1+k} t^{1-{1\over \alpha}} \partial_t v+{p\over \alpha}G_{i, k}^{p,\alpha}\left(x, t, v, \nabla v, \nabla^2 v\right)=0 &\text{in $Q_i$,}\label{general eq-k}\\
v=0 &\text{on $\partial Q_i$.}\label{bdry2}
\end{numcases}
Since we assume that a comparison principle holds for sub- and supersolutions of \eqref{general eq}--\eqref{bdry} that are positive in $\Omega\times (0, \infty)$,  Lemma \ref{lem transform} implies that positive sub- and supersolutions of \eqref{general eq-k}--\eqref{bdry2} also enjoy a comparison principle (which is what we truly need).

When $p\leq0$,   in place of  \eqref{bdry2}, 
$v_i$ satisfies a blow-up boundary and initial condition (precisely $v_i\to-\infty$ for $p=0$, while $v_i\to+\infty$ when $p<0$ on $\partial Q_i$), which enter into the case of state constraints boundary conditions. Then we have to go back to $u_i$ and use the comparison principle for the problem satisfied by $u_\lambda$.

We conclude this section by pointing out the equivalence between
\eqref{time monotone2} and the condition 
\begin{equation}\label{time monotone}
v_i(x, t)\geq v_i(x, s)\quad \text{for any $x\in \Omega_i$, $t\geq s\geq 0$, and $i=1, \ldots, m$}.
\end{equation}
The monotonicity with respect to time will be used in the proof of Theorem \ref{thm main}.

\section{The Minkowski Convolution}
\subsection{Achievability in the interior}

For any given $\lambda\in \Lambda$ and $(x, t)=(\hat{x}, \hat{t})$, we show that the supremum in \eqref{envelope0} can be attained at some $(x_i, t_i)\in Q_i$ for $i=1, 2, \ldots, m$. Our proof is essentially the same of \cite[Lemma 3.1]{IshSa3}. We give the details for the sake of completeness. 

\begin{lem}[Interior maximizers for the envelope]\label{lem achievability}
Suppose that the assumptions of Theorem~{\rm\ref{thm main}} hold. 
Then for any $(\hat{x}, \hat{t})\in Q_\lambda$, there exist $(x_1, t_1)\in {Q}_1$, $(x_2, t_2)\in {Q}_2$, \ldots, $(x_m, t_m)\in {Q}_m$ such that 
\begin{equation}\label{achieve0}
\hat{x}=\sum_{i=1}^m \lambda_i x_i, \ (\hat{t})^\alpha=\sum_{i=1}^m \lambda_i t_i^\alpha,
\end{equation}
and 
\begin{equation}\label{achieve1}
U_{p,\lambda}(\hat{x}, \hat{t})=\left(\sum_{i=1}^m \lambda_i u_i^p(x_i, t_i)\right)^{1\over p}.
\end{equation}
\end{lem}

\begin{proof}
Let us only discuss the case $p>0$, since the results with $p<0$ or $p=0$ clearly hold. 
In view of the compactness of the set 
\[
\left\{(y_1, s_1, y_2, s_2, \ldots, y_m, s_m)\in \prod_{i=1}^m \ol{Q}_i: \hat{x}=\sum_{i=1}^m \lambda_i y_i, \ (\hat{t})^\alpha=\sum_{i=1}^m \lambda_i s_i^\alpha\right\}
\]
and the continuity of 
\[
(y_1, s_1, y_2, s_2, \ldots, y_m, s_m)\mapsto \left(\sum_{i=1}^m \lambda_i u^p_i(x_i, t_i)\right)^{1\over p},
\]
we can find $(x_i, t_i)\in \ol{Q}_i$ for $i=1, 2, \ldots, m$ such that \eqref{achieve0} and \eqref{achieve1} hold. 

Let $\hat{\tau}=(\hat{t})^\alpha$ and $\tau_i=t_i^\alpha$ and recall that $v_i$ is given by \eqref{change variable}. We have 
 \begin{equation}\label{convex comb}
\sum_{i=1}^m \lambda_i x_i=\hat{x}, \quad \sum_{i=1}^m \lambda_i \tau_i=\hat{\tau},
\end{equation}
\begin{equation}\label{convex comb0}
U_{p,\lambda}(\hat{x}, \hat{t})^p=V_{p,\lambda}(\hat{x}, \hat{\tau})=\sum_{i=1}^m \lambda_i v(x_i, \tau_i).
\end{equation}
It thus suffices to show that $(x_i, \tau_i)\in {Q}_i$ for all $i=1, 2, \ldots, m$.

 Assume by contradiction that $(x_i, \tau_i)\in\partial Q_i$ for some $i=1, 2, \ldots, m$.  We derive a contradiction in the following two cases. 

\noindent \textbf{Case 1.}  Suppose that $(x_i, \tau_i)\in \partial Q_i$ for all $i=1, 2, \ldots, m$, 
then by \eqref{bdry} and \eqref{achieve1} we have $U_{p,\lambda}(\hat{x}, \hat{t})=0$, which is a contradiction, since $U_{p,\lambda}(\hat{x}, \hat{t})>0$ for every $(\hat{x}, \hat{t})\in Q_\lambda$.

\noindent \textbf{Case 2.} Assume, without loss of generality, that $(x_1, \tau_1)\in \partial Q_1$ and $(x_2, \tau_2)\in Q_2$. 
Take $\rho\in (0, 1)$ and put 
\[
\tilde{x}_1=x_1+{\rho\over \lambda_1} \tilde{\nu}_1, 
\quad \tilde{x}_2=x_2-{\rho\over \lambda_2}\tilde{\nu}_1, 
\quad \tilde{x}_i=x_i \quad (i=3, 4, \ldots, m),
\]
\[
\tilde{\tau}_1=\tau_1+\mu(t_1){\rho\over \lambda_1}, 
\quad \tilde{\tau}_2=\tau_2-\mu(t_1){\rho\over \lambda_2}, 
\quad \tilde{\tau}_i=\tau_i \quad (i=3, 4, \ldots, m).
\]
Then it is clear that $\sum_i \lambda_i\tilde{x}_i=\sum_i\lambda_i x_i=\hat{x}$, $\sum_i\lambda_i \tilde{\tau}_i=\sum_i\lambda_i \tau_i=\hat{\tau}$. 
By taking $\rho>0$ small enough we also have $(\tilde{x}_1, \tilde{\tau}_1)\in Q_1, (\tilde{x}_2, \tilde{\tau}_2)\in Q_2$. 

 Adopting the local Lipschitz regularity of $u_2$, we get $M>0$ and $\delta_1>0$ such that 
\[
|\nabla v_2|+|\partial_t v_2|\leq M \quad \text{a.e. in $B_{\delta_1}(x_2)\times (\tau_2-\delta_1, \tau_2+\delta_1)\subset Q_2$.}
\]
It follows that 
\begin{equation}\label{achieve2}
\lambda_2 v_2(\tilde{x}_2, \tilde{\tau}_2)-\lambda_2 v_2(x_2, \tau_2)\geq -\lambda_2 M\left(|\tilde{x}_2-x_2|+ |\tilde{\tau}_2-\tau_2|\right)\geq -2M\rho. 
\end{equation}

On the other hand, the condition \eqref{growth behavior} implies that
\[
v_1(\tilde{x}_1, \tilde{\tau}_1)-v_1(x_1, \tau_1)=u_1\left(x_1+\tilde{\nu}_1(x_1){\rho\over \lambda_1}, t_1+\mu(t_1)\left({\rho\over \lambda_1}\right)^{1\over \alpha}\right)^p\geq (2M+1){\rho\over \lambda_1},
\]
which yields that 
\begin{equation}\label{achieve3}
\lambda_1v(\tilde{x}_1, \tilde{t}_1)-\lambda_1 v(x_1, t_1)\geq (2M+1)\rho
\end{equation}
when $\rho$ is sufficiently small. By \eqref{achieve2} and \eqref{achieve3}, we have 
\[
\begin{aligned}
\sum_{i} \lambda_i v_i(\tilde{x}_i, \tilde{\tau}_i)&\geq \lambda_1\left(v_1(\tilde{x}_1, \tilde{\tau}_1)- v_1(x_1, \tau_1)\right)+\lambda_2\left(v_2(\tilde{x}_2, \tilde{\tau}_2)-v_2(x_2, \tau_2)\right)+\sum_i \lambda_i v_i(x_i, t_i)\\
&>\sum_{i=1}^m \lambda_i v_i(x_i, \tau_i)=U_{p,\lambda}(\hat{x}, \hat{t})^p,
\end{aligned}
\] 
which contradicts \eqref{convex comb0}. 
\end{proof}

\subsection{A key lemma}
To show our main result, instead of using $U_{p,\lambda}$ defined in \eqref{envelope0}, we consider the Minkowski convolution $V_{p,\lambda}$ for $v_i$ as given in \eqref{envelope2}.  


It turns out that the following lemma plays a central role in the proof of Theorem \ref{thm main}. 

\begin{lem}[Minkowski convolution preserves subsolutions]\label{lem concave}
 Fix $\lambda\in \Lambda_m$.  Assume that $\Omega_i$ is a bounded smooth domain in $\R^n$ for any $i=1, 2, \ldots, m$. Let $\Omega_\lambda$ be the Minkowski combination as defined in \eqref{minkowski domain}. Assume that $F_i$ satisfies \eqref{f0} for all $i=\lambda, 1, \ldots, m$. Let $0<\alpha\leq 1$ and $p<1$. Suppose that there exists $k\in \R$ such that {\rm(H1)} and {\rm(H2)} hold, where $G_{i, k}^{p,\alpha}$ is given by \eqref{operator-k}. Then:
 \begin{itemize}
\item \underline{Case $0\leq p<1$}. 
\smallskip

Let $v_i$ be a nondecreasing in time upper semicontinuous subsolution of \eqref{general eq-k} for every $i=1, 2, \ldots, m$. 
Suppose that for any fixed $(\hat{x}, \hat{t})\in Q$, the supremum in the definition of $V_{p,\lambda}$ in \eqref{envelope2} at $(\hat{x}, \hat{\tau})$ is attained at some $(x_i, \tau_i)\in Q_i$  for $i=1, 2, \ldots, m$, in other words,  \eqref{convex comb} and \eqref{convex comb0} hold. Then $V_{p,\lambda}$ satisfies the subsolution property for \eqref{general eq-k} at $(\hat{x}, \hat{t})$.  

 \medskip

\item \underline{Case $p<0$}.
\smallskip

Let $v_i$ be a nonincreasing in time lower semicontinuous supersolution of \eqref{general eq-k} for every $i=1, 2, \ldots, m$. 
Suppose that for any fixed $(\hat{x}, \hat{t})\in Q$, the infimum in the definition of $V_{p,\lambda}$ in \eqref{envelope2} at $(\hat{x}, \hat{\tau})$ is attained at some $(x_i, \tau_i)\in Q_i$  for $i=1, 2, \ldots, m$, in other words,  \eqref{convex comb} and \eqref{convex comb0} hold. Then $V_{p,\lambda}$ satisfies the supersolution property for \eqref{general eq-k} at $(\hat{x}, \hat{t})$.   
\end{itemize}
\end{lem}

Using Lemma \ref{lem concave}, we can complete the proof of Theorem \ref{thm main}. 

\begin{proof}[Proof of Theorem~{\rm\ref{thm main}}]
Fix $\lambda\in \Lambda$ arbitrarily. Let $U_{p,\lambda}$, $v_i$, and $V_{p,\lambda}$ 
be given respectively by \eqref{envelope0}, \eqref{change variable}, and \eqref{envelope2} ($i=1, 2, \ldots, m$). 
Adopting Lemma \ref{lem transform}, we can show that if $p\geq 0$ then $v_i$ is a subsolution of \eqref{general eq-k}, while for $p<0$ it is a supersolution of \eqref{general eq-k} for any $i=1, 2, \ldots, m$. 

For any $(\hat{x}, \hat{t})\in Q$, by Lemma \ref{lem achievability} we see that the maximizers in the definition of $U_{p,\lambda}(\hat{x}, \hat{t})$ appear in $Q_i$ for all $i=1, 2, \ldots, m$. This enables us to apply Lemma \ref{lem concave} with $\hat{\tau}=\hat{t}^{1/\alpha}$ to deduce that $V_{p,\lambda}$ is a subsolution or a supersolution, according to the value of $p$, of \eqref{general eq-k} with $i=\lambda$. Adopting Lemma \ref{lem transform} again yields that $U_{p,\lambda}$ is a subsolution of \eqref{general eq} with $i=\lambda$. 
\end{proof}

We next present a proof of of Lemma \ref{lem concave}.

\begin{proof}[Proof of Lemma~{\rm\ref{lem concave}}] 
Let us present the proof in details in the case $p>0$. The cases $p=0$ and $p<0$ can be treated similarly.

Suppose that $\phi\in C^2(\overline{Q})$ is a test function of $V_{p,\lambda}$ at $(\hat{x}, \hat{\tau})\in Q_\lambda$, that is, 
\[
(V_{p,\lambda}-\phi)(x, \tau)\leq (V_{p,\lambda}-\phi)(\hat{x}, \hat{\tau})=0
\] 
for all $(x, \tau)\in \overline{Q}_0$. Due to the maximality of 
\begin{equation}\label{maximality2}
(y_1, s_1, \ldots, y_m, s_m)\mapsto \sum_{i=1}^m \lambda_i v_i(y_i, s_i)-V_{p,\lambda}\left(\sum_{i=1}^m \lambda_i y_i, \sum_{i=1}^m\lambda_i s_i\right)
\end{equation}
over $\prod_{i=1}^m \ol{Q}_i$  at $(x_1, \tau_1, \ldots, x_m, \tau_m)\in \prod_{i=1}^m Q_i$, we see that 
\begin{equation}
(y_1, s_1, \ldots, y_m, s_m)\mapsto \sum_{i=1}^m \lambda_i v_i(x_i, s_i)-\phi\left(\sum_{i=1}^m \lambda_i y_i, \sum_{i=1}^m\lambda_i s_i\right)
\end{equation}
also attains a maximum over $\prod_{i=1}^m \ol{Q}_i$ at $(x_1, \tau_1, \ldots, x_m, \tau_m)$.

We next apply the Crandall-Ishii lemma \cite{CIL}: 
For any $\vep>0$, there exist $(\eta_i, \xi_i, A_i)\in \overline{P}^{2, +} v_i(x_i, \tau_i)$ $(i=1, 2, \ldots, m)$ such that 
\begin{equation}\label{time jet}
\eta_i=\partial_t \phi(\hat{x}, \hat{\tau}),
\end{equation}
\begin{equation}\label{gradient jet}
\xi_i=\nabla \phi(\hat{x}, \hat{\tau}),
\end{equation}
\begin{equation}\label{hessian jet}
\begin{pmatrix}
\lambda_1 A_1 & & & \\
& \lambda_2 A_2 & & \\
& & \ddots & \\
& & & \lambda_{m} A_{m}
\end{pmatrix}
\leq Z+\vep Z^2,
\end{equation}
where $Z$ is given by 
\begin{equation}
Z=\begin{pmatrix}
\lambda_1^2 B & \lambda_1\lambda_2 B & \cdots & \lambda_1\lambda_{m} B\\
\lambda_2 \lambda_1 B & \lambda_2^2 B & \cdots & \lambda_2 \lambda_{m} B\\
\vdots & \vdots & \ddots & \vdots\\
\lambda_{m}\lambda_1 B & \lambda_{m-1} \lambda_2 B &  \cdots & \lambda_{m}^2 B
\end{pmatrix}
\end{equation}
and $B=\nabla^2 \phi(\hat{x}, \hat{t})$. It follows that there exists $C>0$ depending on $\|B\|$ and $\lambda$ such that 
\begin{equation}\label{hessian jet2}
\begin{pmatrix}
\lambda_1 \tilde{A}_1 & & & \\
& \lambda_2 \tilde{A}_2 & & \\
& & \ddots & \\
& & & \lambda_{m} \tilde{A}_{m}
\end{pmatrix}
\leq Z,
\end{equation}
where $\tilde{A}_i=A_i-C\vep I$ for $i=1, 2, \ldots, n+2$.

Adopting the time monotonicity together with \eqref{time jet}, we have 
\begin{equation}\label{time jet2}
\eta_i=\partial_t \phi(\hat{x}, \hat{\tau})\geq 0 \quad \text{for all $i=1, 2, \ldots, m$.}
\end{equation}

Let us consider two different cases. 

\noindent \textbf{Case 1.} Suppose that $\nabla \phi(\hat{x}, \hat{\tau})\neq 0$. 
Then, applying the definition of subsolutions of \eqref{general eq-k}, we have 
\begin{equation}\label{sub-inequality1 rig}
v_i(x_i, \tau_i)^{{1\over p}-1+k}\tau_i^{1-{1/\alpha}}\eta_i+{p\over \alpha}G_{i, k}^{p,\alpha}\left(x_i, \tau_i, v_i(x_i, \tau_i), \xi_i,  A_i\right)\leq 0.
\end{equation}
Multiplying \eqref{sub-inequality1 rig} by $\lambda_i$ and summing up the inequalities, we are led to
\[
\sum_i \lambda_i v_i(x_i, \tau_i)^{{1\over p}-1+k}\tau_i^{1-{1\over \alpha}}\eta_i+{p\over \alpha}\sum_i \lambda_i G_{i, k}^{p,\alpha}\left(x_i, \tau_i, v_i(x_i, \tau_i), \xi_i,  A_i\right)\leq 0,
\]
which by \eqref{time jet} yields that 
\begin{equation}\label{sub-inequality2 rig}
\left(\sum_i \lambda_i v_i(x_i, \tau_i)^{{1\over p}-1+k} \tau_i^{1-{1\over \alpha}}\right) \partial_t \phi\left(\hat{x}, \hat{\tau}\right)+{p\over \alpha}\sum_i \lambda_i G_{i, k}^{p,\alpha}\left(x_i, \tau_i, v_i(x_i, \tau_i),  \xi_i, A_i\right)\leq 0.
\end{equation}
By $\text{(H1)}$ we can easily verify that 
the function $(r, t)\mapsto r^{{1\over p}-1-k} t^{1-{1\over \alpha}}$ 
is convex in $(0, \infty)^2$, which implies that 
\begin{equation}\label{time convexity rig}
\begin{aligned}
\sum_i \lambda_i v_i(x_i, \tau_i)^{{1\over p}-1+k} \tau_i^{1-{1\over \alpha}}&\geq \left(\sum_i \lambda_i v(x_i, \tau_i)\right)^{{1\over p}-1+k}\left(\sum_i \lambda_i \tau_i\right)^{1-{1\over \alpha}}\\
&=V_{p,\lambda}^{{1\over p}-1+k}(\hat{x}, \hat{\tau})\hat{\tau}^{1-{1\over \alpha}}.
\end{aligned}
\end{equation}
The last equality is due to \eqref{convex comb} and \eqref{convex comb0}.  
Using \eqref{time jet2}, \eqref{sub-inequality2 rig}, and \eqref{time convexity rig}, we thus obtain that
\begin{equation}\label{sub-inequality3 rig}
V_{p,\lambda}(\hat{x}, \hat{\tau})^{{1\over p}-1+k} \hat{\tau}^{1-{1\over \alpha}}
 \partial_t \phi\left(\hat{x}, \hat{\tau}\right)+{p\over \alpha}\sum_i \lambda_i G_{i, k}^{p,\alpha}\left(x_i, \tau_i, v_i(x_i, \tau_i),  \xi_i, A_i\right)\leq 0.
\end{equation}
We next apply $\text{(H2)}$ with $X_i=\tilde{A}_i=A_i-C\vep I$ and $Y=B$ to deduce that 
\[
\sum_i \lambda_i G_{i, k}^{p,\alpha}\left(x_i, \tau_i, v_i(x_i, \tau_i), \xi_i,  A_i-C\vep I\right)\geq G_{\lambda,k}^{p,\alpha}\left(\hat{x}, \hat{\tau}, V_{p,\lambda}(\hat{x}, \hat{\tau}), \xi_i, B\right).
\]
It follows from the continuity of $F_i$ (and therefore of $G_{i, k}^{p,\alpha}$) in $\ol{Q}_i\times (0, \infty)\times (\R^n\setminus \{0\})\times \S^n$ that 
\begin{equation}\label{sub-inequality4 rig}
\sum_i \lambda_i G_{i, k}^{p,\alpha}\left(x_i, \tau_i, v(x_i, \tau_i), \xi_i,  A_i\right) \geq G_{\lambda,k}^{p,\alpha}\left(\hat{x}, \hat{\tau}, V_{p,\lambda}(\hat{x}, \hat{\tau}), \xi_i, B\right)-\omega_F(\vep),
\end{equation}
where $\omega_F$ denotes a modulus of continuity describing the locally uniform continuity of $F_i$. 

Plugging \eqref{sub-inequality4 rig} into \eqref{sub-inequality3 rig}, we get 
\[
V_{p,\lambda}(\hat{x}, \hat{\tau})^{{1\over p}-1+k}\hat{\tau}^{1-{1\over \alpha}} \partial_t \phi\left(\hat{x}, \hat{\tau}\right)+{p\over \alpha}G_{\lambda,k}^{p,\alpha}\left(\hat{x}, \hat{\tau}, V_{p,\lambda}(\hat{x}, \hat{\tau}), \xi_i, B\right)\leq {p\over \alpha}\omega_F(\vep),
\]
which yields, by letting $\vep\to 0$, that
\[
V_{p,\lambda}(\hat{x}, \hat{\tau})^{{1\over p}-1+k} \hat{\tau}^{1-{1\over \alpha}} \partial_t \phi\left(\hat{x}, \hat{\tau}\right)+{p\over \alpha}G_{\lambda,k}^{p,\alpha}\left(\hat{x}, \hat{\tau}, V_{p,\lambda}(\hat{x}, \hat{\tau}), \nabla \phi(\hat{x}, \hat{\tau}), \nabla^2 \phi(\hat{x}, \hat{\tau})\right)\leq 0.
\]

\noindent \textbf{Case 2.} Suppose that $\nabla \phi(\hat{x}, \hat{t})=0$. 
We are able to apply Definition \ref{def2} for $V_{p,\lambda}$ by assuming $\nabla^2\phi(\hat{x}, \hat{t})=0$, which by \eqref{hessian jet} further yields that  $A_i\leq 0$ for all $i=1, 2, \ldots, m$. Using Definition \ref{def2} for $v_i$ and the ellipticity of $G_{i, k}^{p,\alpha}$ with $i=1, 2, \ldots, m$ along with \eqref{singularity limit1} and \eqref{operator-k0}, we then have 
\[
v_i(x_i, \tau_i)^{{1\over p}-1+k}\tau_i^{1-{1/\alpha}}\eta_i+{p\over \alpha} \tilde{h}_i\left(x_i, \tau_i, v_i(x_i, \tau_i)\right)\leq 0.
\]
Multiplying this inequality by $\lambda_i$ and summing up over $i=1, 2, \ldots, n+2$, we deduce that 
\[
\left(\sum_i \lambda_i v_i(x_i, \tau_i)^{{1\over p}-1+k} \tau_i^{1-{1\over \alpha}}\right) \partial_t \phi\left(\hat{x}, \hat{\tau}\right)+{p\over \alpha} \sum_i \lambda_i \tilde{h}_i\left(x_i, \tau_i, v(x_i, \tau_i)\right)\leq 0.
\]
Thanks to \eqref{time convexity rig} again and \eqref{singularity convexity}, we may use \eqref{convex comb} and \eqref{convex comb0} to conclude that
\[
V_{p,\lambda}(\hat{x}, \hat{\tau})^{{1\over p}-1+k} \hat{\tau}^{1-{1\over \alpha}} \partial_t \phi\left(\hat{x}, \hat{\tau}\right)+{p\over \alpha} \tilde{h}_0(\hat{x}, \hat{\tau}, V_{p,\lambda}(\hat{x}, \hat{\tau}))\leq 0.
\]
The proof of the case $p>0$  is now complete. As we mentioned at the beginning, the proof for the cases $p=0$ and $p<0$ can be done similarly and we leave the details to the reader. In the latter case, several inequalities need to be changed; for example, \eqref{time jet2} should be reverted and \eqref{hessian jet2} will become 
\[
\begin{pmatrix}
\lambda_1 \tilde{A}_1 & & & \\
& \lambda_2 \tilde{A}_2 & & \\
& & \ddots & \\
& & & \lambda_{m} \tilde{A}_{m}
\end{pmatrix}
\geq Z,
\]
where $\tilde{A}_i=A_i+C\vep I$ for $i=1, 2, \ldots, n+2$  this time. 
\end{proof}



\section{Applications}\label{sec:ex}

Let us discuss applications of Theorem \ref{thm main}, Theorem \ref{thm main1}, and Corollary \ref{cor main2} in this section. We will mainly verify $\text{(H1)}$ and $\text{(H2)}$ in Theorem \ref{thm main} for various concrete examples of $F_i$.  Most of our examples below satisfy the assumptions along with the conditions $1/2\leq \alpha\leq 1$ and $p<1$. 

\subsection{The Laplacian}\label{sec:ex1}
We are able to use Theorem \ref{thm main} and Theorem \ref{thm main1} to recover the main results in \cite{IshSa2, IshSa3}.  Let us first consider Theorem \ref{thm main} when $p\neq 0$ and
\[
F_i(x, t, r, \xi, X)=-\tr X-f_i(x, t, r, \xi),
\]
where we assume that $f_i\geq 0$ and \eqref{nonlinearity cond0} holds for $g_i$ given in \eqref{nonlinearity tran}. 

Taking $k=3-1/p$,  we see that $\text{(H1)}$ holds for any $1/2\leq \alpha\leq 1$ and $0\neq p<1$. 
We can verify $\text{(H2)}$ in this case with $k=3-1/p$. Since $G_{i, k}^{p,\alpha}$ is given by \eqref{heat op tran} and \eqref{nonlinearity cond0} holds, it suffices to show that 
\begin{equation}\label{heat cond}
\sum_i \lambda_i H(r_i, X_i)\geq  H\left(\sum_i \lambda_i r_i, Y\right)
\end{equation}
for $(x_i, t_i, r_i, X_i)\in  Q_i\times (0, \infty)\times \S^n$ ($i=1, 2, \ldots, m$) satisfying \eqref{vertex} and \eqref{key2}, where
\[
H(r, X)= -{1\over p}{r^2}\tr X,  \quad {(r, X)\in (0, \infty)\times \S^n}.
\]

In fact, multiplying both sides of \eqref{key2} by $(r_1\eta, \ldots, r_{m}\eta)\in \R^{mn}$ for an arbitrary $\eta\in \R^n$  from left and right, we have 
\[
\sgn(p)\sum_i \lambda r_i^2 \la X_i \eta, \eta\ra\leq \sgn(p)\left(\sum_i \lambda_i r_i\right)^2 \la Y\eta, \eta\ra;
\]
in other words, 
\[
\sgn(p)\sum_i \lambda r_i^2  X_i\leq  \sgn(p)\left(\sum_i \lambda_i r_i\right)^2 Y. 
\]
Here $\sgn(p)$ denotes the sign of $p\in \R$.  This immediately implies that 
\[
-{1\over p} \sum_i \lambda r_i^2  \tr X_i\geq  -{1\over p} \left(\sum_i \lambda_i r_i\right)^2 \tr Y, 
\]
which is equivalent to \eqref{heat cond}.

We  can further use Corollary \ref{cor main2} to obtain the parabolic power concavity of the solution. 
Since the operator $\ol{G}$ defined by \eqref{gbar} in this case is
\[
\ol{G}(x, t, r, \xi, X)=-{r^2\over p}\tr X- {(1-p)\over p^2 } r|\xi|^2-r^{3-{1\over p}} f\left(x, t^{1\over \alpha}, r^{1\over p}, {1\over p}r^{{1\over p}-1} \xi\right), 
\]
the assumption (H2b) in Corollary \ref{cor main2} requires concavity of \eqref{concave nonlinearity} in $Q_\lambda\times (0, \infty)$. 

We remark that, although the case of the Laplacian has been of course largely and deeply investigated, negative power concavity has never been considered before, to our knowledge. 

We can treat the case $p=0$ in an analogous way. 
When we apply Theorem \ref{thm main} in this case, since \eqref{par cond} in $\text{(H1)}$ is not required, we can choose $k\in \R$ according to the given nonlinear terms $f_i$ in order to guarantee $\text{(H2)}$. 
We may take $k=1$ provided that \eqref{nonlinearity cond0} holds with $g_i$ given by \eqref{nonlinearity tran0}. With such a choice, we can follow the argument in the case $p>0$ to verify \eqref{heat cond} under \eqref{vertex} and \eqref{key2}, where this time we take 
\[
H(r, X)= -e^{2r}\tr X\quad\text{for}\quad {(r, X)\in (0, \infty)\times \S^n}.
\]

\subsection{The normalized $q$-Laplacian}\label{sec:ex2}
We can apply our results to the normalized $q$-Laplacian operator with $1< q<\infty$. Suppose that $F_i$ is given by
\begin{equation}\label{general p-lap}
F_i(x, t, r, \xi, X)=-\tr\left[\left(I+(q-2){\xi\otimes \xi\over |\xi|^2}\right)X\right]-f_i(x, t, r, \xi), 
\end{equation}
where $1< q< \infty$ and $f_i\geq 0$ $(i=\lambda, 1, \ldots, m)$. 
We take $k=3-1/p$ and assume that \eqref{nonlinearity cond0} holds for $g_i$ in \eqref{nonlinearity tran}. 
Suppose that $1/2\leq \alpha\leq 1$ and $0\neq p<1$.
Let us verify the assumption $\text{(H2)}$ with $p\neq 0$ again in this case. 

Similar to the case $q=2$ in Section \ref{sec:ex1}, the key is to prove that for any fixed $\xi\in \R^n\setminus\{0\}$, 
\begin{equation}\label{p-laplacian cond}
\sum_i \lambda_i H_q(r_i, \xi, X_i)\geq  H_q\left(\sum_i \lambda_i r_i, \xi,  Y\right)
\end{equation}
holds for any $(x_i, t_i, r_i, X_i)\in  Q_i\times (0, \infty)\times \S^n$ ($i=1, 2, \ldots, m$) satisfying \eqref{vertex} and \eqref{key2}, where
\[
H_q(r, \xi, X)=-{1\over p}{r^2}\tr\left[\left(I+(q-2){\xi\otimes \xi\over |\xi|^2}\right)X\right]
\]
for $(r, \xi, X)\in (0, \infty)\times (\R^n\setminus \{0\})\times\S^n$. 
To see this, we first notice that
\[
M(\xi):=I+(q-2){\xi\otimes \xi\over |\xi|^2}
\]
is a positive semi-definite matrix in $\S^n$. We thus can write 
\[
H_q(r, \xi, X)=-{1\over p}{r^2}\tr \left(M^{1\over 2}(\xi) X M^{1\over 2}(\xi) \right),
\]
where $M^{1/2}$ is the (nonnegative) square root of $M$.
If \eqref{key2} holds, then by multiplying \eqref{key2} by $\left(r_1 M^{1/2}(\xi)\eta, r_2 M^{1/2}(\xi)\eta, \ldots, r_m M^{1/2}(\xi)\eta\right)\in \R^{mn}$ from both sides for any $\eta\in \R^n$ we can obtain
\[
\sgn(p)\sum_i \lambda_i r_i^2 \tr (M(\xi)X_i)\leq \sgn(p)\left(\sum_i \lambda_i r_i\right)^2 \tr(M(\xi)Y),
\]
which immediately yields the desired property \eqref{p-laplacian cond} for $H_q$.

In this case we also have the parabolic power concavity result in Corollary \ref{cor main2} provided that  \eqref{concave nonlinearity} is concave in $Q_\lambda\times (0, \infty)$ for any $\xi\in \R^n$. The operator $\ol{G}$ in \eqref{gbar} is now given by 
\begin{equation}\label{general p-lap2}
\begin{aligned}
\ol{G}(x, t, r, \xi, X)=-{r^2\over p}\tr\left[\left(I+(q-2){\xi\otimes \xi\over |\xi|^2}\right)X\right] & +{r\over p^2}(1-p)(1-q)|\xi|^2 \\
&-r^{3-{1\over p}} f\left(x, t^{1\over \alpha}, r^{1\over p}, {1\over p}r^{{1\over p}-1}\xi\right). 
\end{aligned}
\end{equation}


  To show that ${\ol{G}}$ verifies (H2b), we again need to assume the concavity of \eqref{concave nonlinearity} in $Q_\lambda\times (0, \infty)$ for any $\xi\in \R^n$.  

We omit the discussion for the case $p=0$, since it can be handled analogously under appropriate assumptions on $f_i$. 

\subsection{General quasilinear operators}\label{sec:ex3} We can further extend the situation in Section \ref{sec:ex2} to more general quasilinear operators in the form of
\[
F_i(x, t, r, \xi, X)=-\tr (A_i(x, \xi) X) -f_i(x, t, r, \xi),
\] 
where $f_i\geq 0$ and $A_i(x, \xi)$ a given nonnegative matrix for any $x\in \Oba_i$ and $\xi\in \R^n\setminus \{0\}$ for all $i=\lambda, 1, \ldots, m$. Let $1/2\leq \alpha\leq 1$ and $p<1$. We assume that $A_i(x, \xi)$ is uniformly continuous and bounded in $\Oba_i\times(\R^n\setminus\{0\})$
for all $i=1, 2, \ldots, m$.

Let us again only consider the case $p\neq 0$. Besides the condition \eqref{nonlinearity cond0} with $g_i$ in \eqref{nonlinearity tran}, the assumption $\text{(H2)}$ with $k=3-1/p$ requires that 
\begin{equation}\label{quasi1}
\sum_{i=1}^m \lambda_i H_{A_i}(x_i, r_i, \xi, X_i)\geq H_{A_\lambda} \left(\sum_{i=1}^m \lambda_i x_i, \sum_{i=1}^m \lambda_i r_i, \xi, Y\right)
\end{equation}
for $(x_i, t_i, r_i, X_i)\in  Q_i\times (0, \infty)\times \S^n$ ($i=1, 2, \ldots, m$) satisfying \eqref{vertex} and \eqref{key2}, where 
\begin{equation}\label{quasi2}
H_{A_i}(x, r, \xi, X)=-{1\over p}r^2 \tr(A_i(x, \xi)X) -{1-p\over p^2} r\la A_i(x, \xi)\xi, \xi\ra
\end{equation}
for $i=\lambda, 1, \ldots, m$. This can be verified easily as in Section \ref{sec:ex2} when all $A_i$ coincide and do not depend on the variable $x$. 

As for the application of Corollary \ref{cor main2}, we see that $\ol{G}$ in this case is given by
\[
\ol{G}(x, t, r, \xi, X)= -r^2 \tr(A(x, \xi)X) -{1-p\over p^2} r\la A(x, \xi)\xi, \xi\ra-g(x, t, r, \xi),
\]
where $g$ is as in \eqref{concave nonlinearity}.

Since the first term on the right hand side can be handled analogously as in Section \ref{sec:ex2}, we omit the details. Hence, a sufficient condition to guarantee the assumption (H2b) is the concavity of  
\[
(x, t, r)\mapsto {1-p\over p^2} r \la A(x, \xi)\xi, \xi\ra+r^{3-{1\over p}} f\left(x, t^{1\over \alpha}, r^{1\over p}, {1\over p}r^{{1\over p}-1} \xi\right)
\]
 in $Q_\lambda\times (0, \infty)$ for any fixed $\xi\neq 0$. In particular, if the coefficient matrix $A$ does not depend on $x$, i.e., $A=A(\xi)$, then we require 
\[
(x, t, r)\mapsto  r^{3-{1\over p}} f\left(x, t^{1\over \alpha}, r^{1\over p}, {1\over p}r^{{1\over p}-1} \xi\right)
\]
is concave for any $\xi\neq 0$, as needed in the previous examples.

We remark that in addition to the normalized $q$-Laplacian discussed in Section \ref{sec:ex2}, applicable quasilinear operators also include the so-called Finsler Laplacian as a special case. Recall that the Finsler-Laplace operator is defined by 
\[
\mathcal{F} u=-\dive(J(\nabla u)\nabla J(\nabla u)),
\]
where $J: \R^n\to \R$ is a given nonnegative convex function of class $C^2(\R^n\setminus\{0\})$ which is positively homogeneous of degree 1, i.e., 
$J(k\xi)=|k| J(\xi)$ for all $k\in \R$ and $\xi\in \R^n$. 
We can write
\[
\mathcal{F} u=-\tr \left(A_J(\nabla u)\nabla^2 u\right),
\quad\mbox{where}\quad A_J(\xi)={1\over 2}\nabla^2 J^2(\xi).
\]
The homogeneity and regularity of the function $J$ imply that the coefficient matrix $A_J$ is bounded and continuous in $\R^n\setminus \{0\}$.

It is now easily seen that Theorem \ref{thm main} does apply to the equations with 
\[
F_i(x, t, r, \xi, X)=-\tr \left(A_J(\xi)X\right)-f_i(x, t, \xi),\quad \text{ $i=\lambda, 1, \ldots, m$.}
\]
Note that the boundedness and continuity of $A_J$ in $\R^n\setminus \{0\}$ enable us to apply the standard viscosity theory to equations involving $\mathcal{F}$; see basic structure assumptions (F1)--(F5) in Appendix \ref{sec:app1} for well-posedness. 

Moreover, since in this case $H_{A_i}$ in \eqref{quasi2} is given by 
\[
H_{A_i}(x, r, \xi, X)=-{1\over p} r^2 \tr(A_J(\xi)X) -{1-p\over p^2} r\la A_J(\xi)\xi, \xi\ra,
\]
for $i=\lambda, 1, \ldots, m$, we can show that \eqref{quasi1} holds for $(x_i, t_i, r_i, X_i)\in  Q_i\times (0, \infty)\times \S^n$ ($i=1, 2, \ldots, m$) satisfying \eqref{vertex} and \eqref{key2}, due to the convexity and nonnegativity of $J$. 

One can use a similar argument to justify the application of Corollary \ref{cor main2} to the Finsler Laplacian.

 
\subsection{The Pucci operator} 
A typical example of fully nonlinear operators is the Pucci operator
\[
\M^-_{a, b}(X)=\inf_{aI \leq A\leq bI} \tr(AX)=a\sum_{e_i\geq 0} ae_i+b\sum_{e_i<0} be_i,
\]
where $0<a\leq b$ are given and $e_i=e_i(X)$ denotes the eigenvalues of any $X\in \S^n$.
 
 Consider 
\begin{equation}\label{pucci1}
F_i(x, t, r, \xi, X)=-\M^-_{a, b}(X)-f_i(x, t, r, \xi)
\end{equation}
for $(x, t)\in \overline{Q}$, $r\in [0, \infty)$, $\xi\in \R^n$, and $X\in \S^n$. As in the examples in Section \ref{sec:ex1} and Section \ref{sec:ex2}, we again assume that $f_i$ is nonnegative and satisfies the relation \eqref{nonlinearity cond0} with $g_i$ defined in \eqref{nonlinearity tran}.

Assume that $1/2\leq \alpha\leq 1$, $p<1$, and $p\neq 0$ so that $\text{(H1)}$ holds with $k=3-1/p$. With such a choice of $k$, we can also verify $\text{(H2)}$. In fact, the operator $G_{i, 3-1/p}$ in this case reads 
\[
G_{i, 3-1/p}(x, t, r, \xi, X)=\sup_{aI\leq A\leq bI} H_A(r, \xi, X)-g_i(x, t, r, \xi),
\]
where 
\[
H_A(r, \xi, X)=- {r^2 \over p}  \tr (AX)-{(1-p)r\over p^2} \la A\xi, \xi\ra.
\]
As shown in Section \ref{sec:ex3}, for any fixed $A\in \S^n$ such that $aI\leq A\leq bI$ and $\lambda\in \Lambda_m$, by \eqref{nonlinearity cond0} we have
\[
\begin{aligned}
\sum_i & \left\{\lambda_i H_A(r_i, \xi, X_i)-\lambda_i g_i(x_i, t_i, r_i, \xi)\right\}\\
& \geq H_A\left(\sum_i \lambda_i r_i, \xi, Y\right)-g_0\left(\sum_i \lambda_i x_i, \sum_i \lambda_i t_i, \sum_i \lambda_i r_i, \xi\right)
\end{aligned}
\]
for any $(x_i, t_i, r_i, X_i)\in  Q_i\times (0, \infty)\times \S^n$ ($i=1, 2, \ldots, m$) satisfying \eqref{vertex}--\eqref{key2}. Maximizing both sides over $aI\leq A\leq bI$, we are led to 
\[
\begin{aligned}
\sum_i & \left\{ \lambda_i \sup_{aI\leq A\leq bI}H_A(r_i, \xi, X_i)-\lambda_i g_i(x_i, t_i, r_i, \xi)\right\}\\
& \geq \sup_{aI\leq A\leq bI} H_A\left(\sum_i \lambda_i r_i, \xi, Y\right)-g_0\left(\sum_i \lambda_i x_i, \sum_i \lambda_i t_i, \sum_i \lambda_i r_i, \xi\right),
\end{aligned}
\]
which completes the verification of $\text{(H2)}$. Similar applications can be obtained in the case $p=0$. One needs to fix $k\in \R$ in accordance with assumptions on $f_i$ ($i=\lambda, 1, 2, \ldots, m$).

We can therefore use Corollary \ref{cor main2} to give a corresponding parabolic power concavity result. Suppose that $f$ is a given nonnegative continuous function and \eqref{concave nonlinearity} is concave with respect to $(x, t, r)$.  Noticing that $\ol{G}$ in \eqref{gbar} in this case is 
\begin{equation}\label{pucci2}
\ol{G}(x, t, r, \xi, X)=\sup_{aI\leq A\leq bI} H_A(r, \xi, X)-g(x, t, r, \xi),
\end{equation}
we can show that it satisfies (H2b). 

We remark that although the result of Theorem \ref{thm main} holds for the operator in \eqref{pucci1}, 
in general, it may not apply to the other type of Pucci operator, which reads
\[
\M^+_{a, b}(X)=\sup_{aI \leq A\leq bI} \tr(AX)=a\sum_{e_i\leq 0} ae_i+b\sum_{e_i>0} be_i, \quad X\in \S^n.
\]
Note that $-\M^-_{a, b}(X)$ is convex in $X$ but $-\M^+_{a, b}(X)$ is concave.

\subsection{Porous medium equation}\label{sec:ex5}
We also show an application of our concavity result to the porous medium equation. Suppose that the equation \eqref{general eq} reduces to 
\[
\partial_t u-\Delta (u^\sigma)=f_i(x, t, u, \nabla u)\quad \text{in $\Omega\times (0, \infty)$}
\]
for a given $\sigma>1$ and $f_i\geq 0$ satisfying assumptions to be specified later.  In this case, the elliptic operator $F_i$ becomes
\[
F_i(x, t, r, \xi, X)=-\sigma r^{\sigma-1}\tr X-\sigma(\sigma-1)r^{\sigma-2}|\xi|^2-f_i(x, t, r, \xi).
\]
The situation for this operator is different from the previous applications. In the case $p\neq 0$, we are actually not able to apply Theorem \ref{thm main} with $k=3-1/p$ or obtain a corresponding concavity result through Corollary  \ref{cor main2}, since our operators hardly satisfy $\text{(H2)}$ no matter what assumption is imposed on $f_i$. Instead, we use Theorem \ref{thm main} with 
\begin{equation}\label{choice k}
k={\sigma\over p}-3
\end{equation}
so as to meet the requirement $\text{(H2)}$.  Note that 
due to the choice of $k$ as in \eqref{choice k}, we have 
\[
G_{i, k}^{p,\alpha}(x, t, r, \xi, X)=-{\sigma\over p} r^2\tr X-{\sigma(\sigma-p)\over p^2} r|\xi|^2-r^{3-{\sigma\over p}}f_i\left(x, t^{1\over \alpha}, r^{1\over p}, {1\over p}r^{{1\over p}-1} \xi\right),
\]
where the major part (the first and second terms) is convex in the sense of $\text{(H2)}$. (They are in fact the same as those major terms in the previous examples.) This is precisely the reason why we chose $k$ as in \eqref{choice k}. 

Moreover, the assumption on $f_i$ is still as in \eqref{nonlinearity cond0} but with 
$g_i$ given by 
\[
g_i(x, t, r)=r^{3-{\sigma\over p}}f_i\left(x, t^{1\over \alpha}, r^{1\over p}, {1\over p}r^{{1\over p}-1} \xi\right).
\]
In order to meet the condition $\text{(H1)}$, we need to additionally assume that 
\begin{equation}\label{porous2}
\text{either }\quad 2p-\sigma+1\leq 0\quad \text{or}\quad {p\over 2p-\sigma+1}\leq \alpha\leq 1.
\end{equation}
We again omit the details for the case $p=0$. 

For the concavity result in Theorem \ref{thm main1}, we can make the same choice of parameters $k$, $\alpha$, and $p$ as in \eqref{choice k} and \eqref{porous2}, and assume that 
\begin{equation}\label{porous1}
(x, t, r)\mapsto r^{3-{\sigma\over p}}f\left(x, t^{1\over \alpha}, r^{1\over p}, {1\over p}r^{{1\over p}-1} \xi\right)
\end{equation}
is concave for any $\xi$.
\begin{rmk}
The condition \eqref{porous2} means that, for any given $\sigma>1$, in order to obtain parabolic power concavity with $p\in (-\infty, 1)$  satisfying $\sigma \geq 2p+1$, we basically have no restrictions on $\alpha\in (0, 1]$ except for \eqref{porous1} when $p\neq 0$ (or a variant of \eqref{porous1} when $p=0$). 
On the other hand, for $p\in (-\infty, 1)$ fulfilling $\sigma<2p+1$, \eqref{porous2} requires that $p\geq \sigma-1$ so as to allow room for $\alpha$; there is no result for the range $(\sigma-1)/2< p<\sigma-1$ no matter what $\alpha$ is taken. 
\end{rmk}

Note that when $\sigma=1$ the conditions \eqref{porous1} and \eqref{porous2} reduce to the assumptions for the Laplacian case as discussed in Section \ref{sec:ex1}. 

\appendix

\section{Related issues on viscosity solutions}\label{sec:app}
Let us discuss several important properties of viscosity solutions, which are used in the previous sections. We first review well-posedness for general fully nonlinear parabolic equations and then give sufficient conditions for the time monotonicity of solutions and a Hopf-type property. 


\subsection{Well-posedness}\label{sec:app1}

 Let $\Omega$ be a bounded smooth domain.  We below consider the equation~\eqref{single eq} in $Q=\Omega\times (0, \infty)$ with a homogeneous initial boundary condition, i.e., 
\begin{numcases}{}
\partial_t u+F(x, t, u, \nabla u, \nabla^2 u)=0 &\text{in $Q$, }\label{eq app}\\
u=0 &\text{on $\partial Q$.}\label{bdry app}
\end{numcases}
We denote by $\nu$ the inward unit normal vector to $\pO$. Set $\tilde{\nu}(x)=\nu(x)$ if $x\in \pO$ and $\tilde{\nu}(x)=0$ if $x\in \Omega$.

Let us impose the following basic structure assumptions on $F$:
\begin{enumerate}
\item[(F1)] $F: \Oba\times [0, \infty)\times \R\times (\R^n\setminus \{0\})\times \S^n\to \R$ is continuous. 
\item[(F2)] $F$ is degenerate elliptic, i.e., \eqref{eq elliptic} holds 
for all $x\in \Oba_i$, $t\in [0, \infty)$, $r\in [0, \infty)$, $\xi\in \R^n\setminus\{0\}$, and $X_1, X_2\in \S^n$ satisfying $X_1\geq X_2$.
\item[(F3)] $F$ is monotone in the sense that there exists $c\in \R$ such that \eqref{eq monotone} holds
for all $(x, t, \xi, X)\in \Oba_i\times [0, \infty)\times (\R^n\setminus\{0\})\times \S^n$ and $r_1, r_2\in [0, \infty)$ satisfying $r_1\leq r_2$.
\item[(F4)] For any $R>0$, there exists a modulus of continuity $\omega_R$ such that 
\[
|F(x, t, r, \xi, X)-F(y, t, r, \xi, X)|\leq \omega_R(|x-y||\xi|+1)
\]
for all $x, y\in \Oba$, $t\in [0, \infty)$, $|r|\leq R$, $\xi\in \R^n\setminus \{0\}$, and $X\in \S^n$.
\item[(F5)] There exists a continuous function $h: \ol{Q}\times [0, \infty)\to \R$ such that
\begin{equation}\label{app f0}
h(x, t, r)=(F)_\ast(x, t, r, 0, 0)=(F)^\ast(x, t, r, 0, 0)\quad \text{for $(x, t, r)\in Q\times (0, \infty)$}.
\end{equation}
\end{enumerate}

Under these assumptions, viscosity solutions (sub- and supersolutions) of \eqref{eq app} are defined as in Section \ref{sec:viscosity}. It is known that the following comparison theorem holds. 

\begin{thm}[Theorem 3.6.1 in \cite{Gbook}]\label{thm comparison}
Assume that $\Omega$ is bounded and {\rm(F1)}--{\rm(F5)} hold. Let $u$ and $v$ be respectively an upper semicontinuous subsolution and a lower semicontinuous supersolution of \eqref{eq app}. If $u\leq v$ on $\partial Q$, then $u\leq v$ in $\ol{Q}$.
\end{thm}
Uniqueness of viscosity solutions to \eqref{eq app}--\eqref{bdry app} is an immediate consequence of Theorem~\ref{thm comparison}. One can obtain existence of a  positive viscosity solution by adopting Perron's method for viscosity solutions if a positive subsolution exists; see precise arguments in \cite{I1} and \cite[Section 4]{CIL} for nonsingular equations and \cite[Section 2.4]{Gbook} for singular one. 

Moreover, a standard argument (\cite[Theorem 2.2.1]{Gbook} for instance) yields that the unique solution $u$ is stable in the sup norm under uniform perturbation of the operator and initial boundary data.

As for the spatial Lipschitz regularity, which is needed in Theorem \ref{thm main}, we refer to relevant results in the literature. Lipschitz or H\"{o}lder regularity of viscosity solutions to fully nonlinear nonsingular parabolic equations is given in 
\cite{Ba0, Wang1, Wang2, KK2, BaSo, Ba3} etc. 
We also consult local Lipschitz estimates for  singular parabolic equations such as the normalized $q$-Laplace equations in \cite{Doe, PRu, JiSi} ($1<q<\infty$) and in \cite{JuKa} ($q=\infty$).

\subsection{Monotonicity in time}\label{sec:app2}
The next two subsections are devoted to discussion on the assumptions (i) and (ii) in Theorem \ref{thm main} (and in Theorem \ref{thm main1}) for $p\in(0,1)$.  Since it is in general quite restrictive to assume (i) and (ii) on $F_i$, 
we consider the approximate equation \eqref{perturb eq}. 
We will actually provide sufficient conditions to guarantee (i) and (ii) for \eqref{perturb eq} instead of \eqref{general eq}.  

Let us first study the time monotonicity in (i). Suppose that
\begin{equation}\label{time monotonicity1}
h(x, 0, 0)=F_\ast(x, 0, 0, 0, 0)=F^\ast(x, 0, 0, 0, 0)\leq 0,
\qquad\qquad\qquad\quad
\end{equation}
\begin{equation}\label{time monotonicity2}
\begin{split}
 & F(x, t, r, p, X)\leq F(x, s, r, p, X)\\
 & \text{for any $t\geq s\geq 0$ and $(x, r, p, X)\in \Oba\times [0, \infty)\times (\R^n\setminus \{0\})\times \S^n$.}
\end{split}
\end{equation}
\begin{lem}[Monotonicity in time]\label{lem time monotone}
Assume that  $\Omega$ is bounded and $F$ satisfies {\rm(F1)}--{\rm(F5)}. If \eqref{time monotonicity1} and \eqref{time monotonicity2} hold,  then the unique solution $u$ of  \eqref{general eq}--\eqref{bdry} satisfies \eqref{time monotone2}. 
\end{lem} 
\begin{proof}
The assumptions \eqref{time monotonicity1} and \eqref{time monotonicity2} imply that the constant zero is a subsolution of \eqref{general eq}--\eqref{bdry}. If follows that $u\geq 0$ in $\overline{Q}$ by the comparison principle. Fix $\tau>0$ arbitrarily and set $w_\tau(x, t)=u(x, t+\tau)$ for all $(x, t)\in \overline{Q}$. Then by \eqref{time monotonicity2} we can easily show that $w_\tau$ is a supersolution of \eqref{general eq}. Since $w_\tau(\cdot, 0)= u(\cdot, \tau)\geq u(\cdot, 0)$ in $\Oba$, we can use the comparison principle again to prove that $w_\tau\geq u$ in $\overline{Q}$, which immediately yields \eqref{time monotone2} due to the arbitrariness of $\tau$. 
\end{proof}

A more specific situation fulfilling \eqref{time monotonicity1}, \eqref{time monotonicity2}, and other assumptions needed in our main results is the case when 
\[
F(x, t, r, \xi, X)=\L(\xi, X)-f(x, t, r)
\]
for all $(x, t, r, \xi, X)\in \ol{Q}\times [0, \infty)\times (\R^n\setminus\{0\})\times \S^n$, where $f$ is nonnegative in $\ol{Q}\times [0, \infty)$ and nondecreasing in $t$, and $\L_\ast(0, 0)=\L^\ast(0, 0)=0$. Concrete examples of $\L$ include the Pucci operator, 
normalized $q$-Laplacian ($1<q\leq \infty$), and more general quasilinear operators as discussed in Section \ref{sec:ex}.

We next discuss the assumption (ii) in Theorem \ref{thm main} and Theorem \ref{thm main1}. Assume that $0<p<1$ for the rest of this section. Note that the condition \eqref{growth behavior} can be divided into two parts. One part is the following growth behavior near the initial moment:
\begin{equation}\label{rapid evolution}
{1\over \rho}u^p\left(x+\tilde{\nu}(x)\rho, \rho^{1/\alpha}\right)\to \infty\quad \text{as $\rho\to 0$ for every $x\in \Oba$.}
\end{equation}
The other part can be expressed by 
\begin{equation}\label{bdry growth}
{1\over \rho} u^p\left(x+\nu(x)\rho, t\right)\to \infty \quad \text{as $\rho\to 0$ for every $x\in \pO$ and $t>0$.}
\end{equation}
We will later see that for $p\in(0,1)$ \eqref{bdry growth} is a consequence of the Hopf lemma; consult Section~\ref{sec:app3}. 

In order to obtain \eqref{rapid evolution}, we need to strengthen the condition \eqref{time monotonicity1} in Lemma \ref{lem time monotone}. 
\begin{lem}[Rapid initial growth]\label{lem time monotone2}
Let $0<p<1$. Assume that  $\Omega$ is bounded and $F$ satisfies {\rm(F1)}--{\rm(F5)}. Assume that \eqref{time monotonicity2} holds.  
Assume that 
\begin{equation}\label{initial monotone cond}
\left\{\begin{aligned}
&\text{there exist $\beta, \beta', t_0>0$ and $\psi_0\in C^2(\Oba)$ with  $\psi_0>0$ in $\Omega$ and $\psi_0=0$ on $\partial \Omega$}\\
&\text{such that}\quad \displaystyle p\beta'+{p \beta \over \alpha}<1,\quad 
\displaystyle \sup_{\Omega}{\dist(\cdot, \pO)^{\beta'}\over \psi_0}<\infty,\quad\text{and} \\
& {\beta}\psi_0(x)t^{\beta-1}+F_\ast\left(x, t, \psi_0(x)t^{\beta}, \nabla \psi_0(x)t^{\beta}, \nabla^2 \psi_0(x)t^{\beta}\right)\leq 0
\quad\text{in $\Omega\times (0, t_0)$}.
\end{aligned}
\right.
\end{equation}
Then the unique solution $u$ of \eqref{general eq}--\eqref{bdry} satisfies \eqref{rapid evolution}.
\end{lem}

\begin{proof}
By the last inequality in \eqref{initial monotone cond} we observe that 
the function~$(x, t)\mapsto\psi_0(x)t^{\beta}$
is a subsolution of  \eqref{general eq} restricted in $\Omega\times(0, t_0)$. Noticing that $\psi_0=0$ on $\partial \Omega$, we can use the comparison principle to obtain that $u(x, t)\geq \psi_0(x)t^{\beta}$ 
for $(x, t)\in \Omega\times (0, t_0)$. 
When $x\in \Omega$, we easily deduce \eqref{rapid evolution}, since $\psi_0>0$ in $\Omega$ and $\beta<\alpha/p$. 

If $x\in \pO$, noticing that $\psi_0(x+\rho \nu(x))\geq c\rho^{\beta'}$ in $\Omega$ for some $c>0$, we have
\[
u^p\left(x+\rho \nu(x), \rho^{1/\alpha}\right)\geq c^p\rho^{p\beta'+p\beta/\alpha},
\] 
which implies \eqref{rapid evolution}, due to the condition that $p\beta'+{p \beta /\alpha}<1$.
\end{proof}

Let us discuss how to apply Lemma \ref{lem time monotone2} 
in our applications under the assumption $\alpha\geq p$. 
Suppose that $F_i$ ($i=\lambda, 1, \ldots, m$) satisfies \eqref{time monotonicity1} and \eqref{time monotonicity2}, i.e., 
\[
h_i(x, 0, 0)=(F_i)\ast(x, 0, 0, 0, 0)=(F_i)^\ast(x, 0, 0, 0, 0)\leq 0,\qquad\qquad\,\,
\]
\begin{equation*}
\begin{split}
 & F_i(x, t, r, p, X)\leq F_i(x, s, r, p, X)\\
 & \text{for any $t\geq s\geq 0$ and $(x, r, p, X)\in \Oba_i\times [0, \infty)\times (\R^n\setminus \{0\})\times \S^n$.}
 \end{split}
 \end{equation*}
As mentioned in the beginning of this section, these assumptions in general may not guarantee \eqref{initial monotone cond} for $F=F_i$. However, we can first turn to study the perturbed equation \eqref{perturb eq} first and then let $\vep\to 0$. In addition to the perturbation for the operators,  we put $p_\vep=p-\vep$ with  $\vep>0$  small so that $\alpha>p_\vep$. We can show \eqref{initial monotone cond} holds for $p=p_\vep$  and $F=F_{i, \vep}$ for any $i=\lambda, 1, \ldots, m$. Indeed, we can choose 
\[
1<\beta<{\alpha\over p_\vep}, \quad 0<\beta'<1-{\beta p_\vep\over \alpha},
\]
and $\psi_0\in C^2(\Oba)$ such that $\psi_0=\dist(\cdot, \pO)^{\beta'}$ near $\pO$. Then we can verify the last inequality in \eqref{initial monotone cond} with $F=F_{i, \vep}$ and $p=p_{\vep}$ provided that $t_0$ is sufficiently small.


\subsection{The Hopf-type property}\label{sec:app3}
We finally discuss the property \eqref{bdry growth}, which is used to derive the condition (ii) of Theorem \ref{thm main}. It is in fact related to the so-called Hopf-type property:
\begin{enumerate}
\item[(HP)] Fix any $x_0\in \partial \Omega$ and $t_0>0$. Assume that there exist $0<\delta<t_0$ and $y_0\in \Omega$  such that \\

\begin{itemize}
\item $B_\delta(y_0)\subset \Omega$ and $\ol{B_\delta(y_0)}\cap \pO=\{x_0\}$;
\item $u$ is a supersolution of \eqref{eq app};
\item $u$ satisfies $u(x, t)>u(x_0, t_0)=0$ for any $(x, t)\in B_\delta(y_0)\times [t_0-\delta, t_0]$.
\end{itemize}
Then 
\begin{equation}\label{eq hopf}
\liminf_{\rho\to 0+}{1\over \rho}u\left(x_0+\rho {y_0-x_0\over |y_0-x_0|}, t_0\right)>0.
\end{equation}
\end{enumerate}
It is obvious that \eqref{bdry growth} is an immediate consequence of \eqref{eq hopf} when $0<p<1$. 
See \cite{DaL, Gri, CLN} for sufficient conditions on $F$ in order to obtain (HP).

For our own purpose in this work, following the same method described in Section \ref{sec:app2}, we use \eqref{eq hopf} for the approximate problem \eqref{perturb eq}, where $F_{i, \vep}$ is the perturbed operator given in \eqref{perturb F}.
It turns out that we still only need \eqref{time monotonicity1} and \eqref{time monotonicity2} on $F=F_i$ to show \eqref{eq hopf} 
for a supersolution of \eqref{perturb eq}. 

Note that \eqref{time monotonicity1} and \eqref{time monotonicity2} 
imply that 
\begin{equation}\label{tech cond}
h_i(x, t, 0)=(F_i)_\ast(x, t, 0, 0, 0)=(F_i)^\ast(x, t, 0, 0, 0)\leq 0 \quad \text{for all $(x, t)\in \ol{Q}_i$.}
\end{equation}
Let $u_{i, \vep}$ be a supersolution of \eqref{perturb eq}. 
Since the constant zero is a subsolution, 
we have $u_{i, \vep}\geq 0$ in $\Oba_i\times [0, \infty)$.   
Denote $\zeta_0=(y_0, t_0)$ and $z=(x, t)$.  
We take 
\[
v_\gamma (z)=e^{-\gamma|z-\zeta_0|^2}-e^{-\gamma \delta^2}
\]
with $\gamma>0$ large. Since 
\[
\sup_{\ol{B_\delta(y_0)}\times [t_0-\delta, t_0+\delta]}\left(|v_\gamma|+|\partial_t v_\gamma|+|\nabla v_\gamma|+|\nabla^2 v_\gamma|\right)\to 0 \qquad \text{as $\gamma\to \infty$, }
\]
it follows from \eqref{tech cond} and (F5) that, when $\gamma>0$ is large,  
\[
\begin{aligned}
& \partial_t v_\gamma(z)+(F_{i, \vep})_\ast(z, v_\gamma(z), \nabla v_\gamma(z), \nabla^2v_\gamma (z))\\
& \leq  \partial_t v_\gamma (z)+ (F_i)_\ast(z, v_\gamma (z), \nabla v_\gamma(z), \nabla^2v_\gamma (z))-\vep\leq -{\vep\over 2}
\end{aligned}
\] 
for any $z\in B_\delta(y_0)\times (t_0-\delta, t_0+\delta)$. 

We have shown that $v_\gamma$ is a subsolution of \eqref{perturb eq} in $B_\delta(y_0)\times (t_0-\delta, t_0)$.  Noticing that
\[
u_{i, \vep}\geq 0\geq  v_\gamma \quad \text{on $\left(\ol{B_\delta(y_0)}\times\{t_0-\delta\}\right)\cup \left(\partial B_\delta(y_0)\times (t_0-\delta, t_0+\delta)\right)$},
\]  
by comparison principle we have 
\[
u_{i, \vep}\geq v_\gamma \quad \text{in $\ol{B_\delta(y_0)}\times [t_0-\delta, t_0+\delta]$, }
\] 
which implies that 
\[
u_{i, \vep}\left(x_0+\rho {y_0-x_0\over |y_0-x_0|}, t_0\right)\geq v_\gamma\left(x_0+\rho {y_0-x_0\over |y_0-x_0|}, t_0\right)\geq \rho \gamma |x_0-y_0| e^{-\gamma|x_0-y_0|^2} +o(\rho).
\]
We thus complete the proof of \eqref{eq hopf} for any supersolution $u_{i, \vep}$ of \eqref{perturb eq}. 

\end{document}